\documentclass[12pt]{amsart}
\usepackage{graphics,color,epsfig}
\usepackage{amsfonts,amssymb,wasysym}

\newcommand{\Q}{\mathbb Q}
\newcommand{\R}{\mathbb R}
\newcommand{\Z}{\mathbb Z}

\newcommand{\ba}{\mathbf{a}}
\newcommand{\bm}{\mathbf{m}}
\newcommand{\bp}{\mathbf{p}}

\newcommand{\bx}{\mathbf{x}}

\newcommand{\cB}{\mathcal B}
\newcommand{\cC}{\mathcal C}
\newcommand{\cE}{\mathcal E}
\newcommand{\cL}{\mathcal L}
\newcommand{\cN}{\mathcal N}
\newcommand{\cU}{\mathcal U}
\newcommand{\cV}{\mathcal V}

\DeclareMathOperator{\conv}{conv}
\DeclareMathOperator{\diag}{diag}
\DeclareMathOperator{\diam}{diam}
\DeclareMathOperator{\dist}{dist}
\DeclareMathOperator{\hor}{hor}
\DeclareMathOperator{\BA}{BA}

\DeclareMathOperator{\DI}{DI}

\DeclareMathOperator{\Hdim}{H.\!dim}
\DeclareMathOperator{\Sing}{Sing}
\DeclareMathOperator{\Span}{span}

\DeclareMathOperator{\SL}{SL}
\DeclareMathOperator{\vol}{vol}

\newcommand{\eps}{\varepsilon}

\newcommand{\zip}{\varnothing}
\newcommand{\ls}{\preceq}
\newcommand{\gs}{\succeq}
\newcommand{\Wedge}{\wedge}

\newtheorem{theorem}{Theorem}[section]

\newtheorem{corollary}[theorem]{Corollary}
\newtheorem{lemma}[theorem]{Lemma}
\newtheorem{proposition}[theorem]{Proposition}
\newtheorem*{conj}{Conjecture}

\theoremstyle{definition}
\newtheorem{definition}[theorem]{Definition}

\newtheorem*{notation}{Notation}

\theoremstyle{remark}

\newtheorem{remark}[theorem]{Remark}

\title{Hausdorff dimension of the set of singular pairs}
\author{Yitwah Cheung}
\address{San Francisco State University \\ 
San Francisco, CA 94132, U.S.A.}
\email{cheung@math.sfsu.edu}
\date{\today}
\thanks{This research is partially supported by NSF grant DMS 0701281 and 
  a Spring 2007 Mini-Grant Award from San Francisco State University.}  

\begin{document}
\begin{abstract}
In this paper we show that the Hausdorff dimension of the set of singular 
  pairs is $\tfrac{4}{3}$.  
We also show that the action of $\diag(e^t,e^t,e^{-2t})$ on $\SL_3\R/\SL_3\Z$ 
  admits divergent trajectories that exit to infinity at arbitrarily slow 
  prescribed rates, answering a question of A.N.~Starkov.  
As a by-product of the analysis, we obtain a higher dimensional generalisation 
  of the basic inequalities satisfied by convergents of continued fractions.  
As an illustration of the techniques used to compute Hausdorff dimension, 
  we show that the set of real numbers with divergent partial quotients has 
  Hausdorff dimension $\tfrac12$.  
\end{abstract}
\maketitle

\section{Introduction}
Let $\Sing(d)$ denote the set of all singular vectors in $\R^d$.  
Recall that $\bx\in\R^d$ is said to be \emph{singular} if 
  for every $\delta>0$ there exists $T_0$ such that for all $T>T_0$ 
  the system of inequalities 
\begin{equation}\label{ieq:sing}
  \|q\bx-\bp\|<\frac{\delta}{T^{1/d}}\quad\text{and}\quad 0<q<T 
\end{equation}
  admits an integer solution $(\bp,q)\in\Z^{d+1}$.  
Since $\Sing(d)$ contains every rational hyperplane in $\R^d$, 
  its Hausdorff dimension is between $d-1$ and $d$.  
In this paper, we prove 
\begin{theorem}\label{thm:main}
  The Hausdorff dimension of $\Sing(2)$ is $\frac{4}{3}$.  
\end{theorem}

Singular vectors that lie on a rational hyperplane are said to 
  be \emph{degenerate}.  
Implicit in this terminology is the expectation that the set 
  $\Sing^*(d)$ of all \emph{nondegenerate} singular vectors is 
  somehow larger than the union of all rational hyperplanes 
  in $\R^d$, which is a set of Hausdorff dimension $d-1$.  
The papers \cite{Ba77}, \cite{Ya87}, \cite{Ry90} and \cite{Ba92} 
  give lower bounds on certain subsets of $\Sing^*(d)$ that, in 
  particular, imply $\Hdim \Sing^*(d)\ge d-1$.  
Theorem~\ref{thm:main} shows that strict inequality holds 
  in the case $d=2$.

\subsection*{Divergent trajectories}
There is a well-known dynamical interpretation of singular vectors.  
Let $G/\Gamma$ be the space of oriented unimodular lattices in $\R^{d+1}$, 
  where $G=\SL_{d+1}\R$ and $\Gamma=\SL_{d+1}\Z$.  
A path $(\Wedge_t)_{t\ge0}$ in $G/\Gamma$ is said to be \emph{divergent} 
  if for every compact subset $K\subset G/\Gamma$ there is a time $T$ 
  such that $\Wedge_t\not\in K$ for all $t>T$.  
By Mahler's criterion, $(\Wedge_t)$ is divergent iff the length of the 
  shortest nonzero vector of $\Wedge_t$ tends to zero as $t\to\infty$.  
It is not hard to see that $\bx\in\R^d$ is singular iff 
  $\ell(g_th_\bx\Z^{d+1})\to0$ as $t\to\infty$, where 
\begin{equation*}
  g_t = \begin{pmatrix}
        e^t &        &     &   \\
            & \ddots &     &   \\
            &        & e^t &   \\
            &        &     & e^{-dt} 
        \end{pmatrix}, \quad 
  h_\bx = \begin{pmatrix}
         1 &        &   & -x_1   \\
           & \ddots &   & \vdots \\
           &        & 1 & -x_d   \\
           &        &   &   1 
        \end{pmatrix}, 
\end{equation*}
  and $\ell(\cdot)$ denotes the length of the shortest nonzero vector.  
Thus, $\bx$ is singular if and only if $(g_th_\bx\Gamma)_{t\ge0}$ is a 
  divergent trajectory of the homogeneous flow on $G/\Gamma$ induced by 
  the one-parameter subgroup $(g_t)$ acting by left multiplication.  

As a corollary of Theorem~\ref{thm:main} we have 
\begin{corollary}\label{cor:main}
Let $D\subset\SL_3\R/\SL_3\Z$ be the set of points that lie on divergent 
  trajectories of the flow induced by $(g_t)$.  Then $\Hdim D=7\frac{1}{3}$.  
\end{corollary}
\begin{proof}
Let $P:=\{p\in G|g_tpg_{-t}\text{ stays bounded as }t\to\infty\}$ and 
  note that every $g\in G$ can be written as $ph_\bx\gamma$ for some 
  $p\in P$, $\bx\in\R^d$ and $\gamma\in\Gamma$.  
Since the distance between $g_tph$ and $g_th$ with respect to any 
  right-invariant metric on $G$ stays bounded as $t\to\infty$, 
  it follows that $$D=\cup_{\bx\in\Sing(d)}Ph_\bx\Gamma.$$ 
Since $P$ is a manifold and $\Gamma$ is countable, the Hausdorff 
  codimension of $D$ in $G/\Gamma$ coincides with that of $\Sing(d)$ 
  in $\R^d$.  
\end{proof}

Further results about singular vectors and divergent trajectories can 
  be found in the papers \cite{KW1}, \cite{KW2}, \cite{W1} and \cite{W2}.  

\subsection*{Related results}
It should be mentioned that the notion of a singular vector is dual 
  to that of a badly approximable vector.  
Recall that $\bx\in\R^d$ is \emph{badly approximable} if there is a $c>0$ 
  such that $\|q\bx-\bp\|>cq^{-1/d}$ for all $(\bp,q)\in\Z^d\times\Z_{>0}$.  
As with $\Sing(d)$, the set $\BA(d)$ of badly approximable vectors in $\R^d$ 
  admits a characterisation in terms of flow on $G/\Gamma$ induced by $(g_t)$.  
Specifically, $\bx\in\BA(d)$ if and only if $(g_th_\bx\Gamma)_{t\ge0}$ is a 
  \emph{bounded trajectory}, i.e. its closure in $G/\Gamma$ is compact.  
In \cite{Sc66} Schmidt showed that $\Hdim\BA(d)=d$, which implies that 
  the set $B\subset G/\Gamma$ of points that lie on bounded trajectories 
  of the flow induced by $(g_t)$ has $\Hdim B=\dim G$.  
In \cite{KM96} Kleinbock-Margulis generalised the latter statement to the 
  setting of non-quasi-unipotent homogeneous flows where 
  $G$ is a connected real semisimple Lie group, $\Gamma$ a lattice in $G$, 
  and $(g_t)$ a one-parameter subgroup such that $\text{Ad}~g_1$ has an 
  eigenvalue of absolute value $\neq1$.  

For divergent trajectories of such flows, Dani showed in \cite{Da85} that 
  if $\Gamma$ is of ``rank one'' then the set $D\subset G/\Gamma$ of points 
  that lie on divergent trajectories of the flow is a countable union of 
  proper submanifolds, implying that its Hausdorff dimension is integral 
  and strictly less than $\dim G$.  
Much less is known in the case of lattices of higher rank.  
In the special case of the reducible lattice $\Gamma=(\SL_2\Z)^n$ in the 
  product space $G=(\SL_2\R)^n$ with the one-paramter subgroup $(g_t)$ 
  inducing the same diagonal flow in each factor, it was shown in \cite{Ch07} 
  that the set $D$ has Hausdorff dimension $\dim G-\frac{1}{2}$, $n\ge2$.  
Corollary~\ref{cor:main} furnishes a higher rank example with $\Gamma$ an 
  irreducible lattice, suggesting the following 
\begin{conj}
For any non-quasi-unipotent flow $(G/\Gamma,g_t)$ on a finite-volume, 
  noncompact homogeneous space, the set $D\subset G/\Gamma$ of points 
  that lie on divergent trajectories has Hausdorff dimension strictly 
  less than $\dim G$.  
\footnote{It is a well-known result due to G.A.~Margulis that every 
quasi-unipotent flow on a finite-volume homogeneous space does not 
admit any divergent trajectories.}
\end{conj}

Similar results are known for Teichm\"uller flows that are consistent 
  with the behavior of non-quasi-unipotent flows; for example, in \cite{Ma92} 
  Masur showed that for any Teichm\"uller disk $\tau$, the set of points that 
  lie on divergent trajectories is a subset $D_\tau\subset\tau$ of Hausdorff 
  codimension at least $\frac{1}{2}$.  
For some $\tau$, the Hausdorff codimension $\frac{1}{2}$ is realised \cite{Ch03}.  
It is also known that for a generic $\tau$, the Hausdorff codimension of $D_\tau$ 
  is strictly less than one \cite{MS91}.  
Likewise, the set $B_\tau\subset \tau$ of points that lie on bounded trajectories 
  of the Teichm\"uller flow is shown in \cite{KW04} to have $\Hdim B_\tau=\dim\tau$.  

\subsection*{Further applications}
Let $\DI_\delta(d)$ be the set of all $\bx\in\R^d$ for which there exists $T_0$ 
  such that for all $T>T_0$ the system of inequalities (\ref{ieq:sing}) admits 
  an integer solution.  
In the course of proving Theorem~\ref{thm:main} we obtain 
\begin{theorem}\label{thm:main:2}
There are positive constants $c_1,c_2$ such that 
  $$\frac{4}{3}+\exp(-c_1\delta^{-4}) \le \Hdim \DI_\delta(2) 
      \le \frac{4}{3} + c_2\delta.$$
\end{theorem}
The upper bounds in Theorems~\ref{thm:main} and \ref{thm:main:2} are obtained 
  in \S\ref{S:upper} while the lower bounds are obtained in \S\ref{S:lower}.  

\begin{corollary}\label{cor:main:2}
There is a compact subset $K\subset\SL_3\R/\SL_3\Z$ such that the set $D_K$ 
  of all points that lie on trajectories of $(g_t)$ that eventually stay 
  outside of $K$ is a set of Hausdorff dimension strictly less than $8$.  
\end{corollary}

Our techniques also allow us to answer a question of Starkov in \cite{St} 
  concerning the existence of divergent trajectories of the flow $g_t$.  
\S\ref{S:Starkov} is devoted to a proof of the following.  
\begin{theorem}\label{thm:main:3}
Given any function $\eps(t)\to0$ as $t\to\infty$ there is a dense set of 
  $\bx\in\Sing^*(2)$ such that $\ell(g_th_\bx\Z^{d+1})\ge\eps(t)$ as $t\to\infty$.  
\end{theorem}

As a by-product of our analysis, we also obtain applications to number theory.  
\begin{theorem}\label{thm:main:4}
Let $\tfrac{\bp_j}{q_j},j=0,1,\dots$ be the sequence of best approximations 
  to $\bx$ relative to some given norm $\|\cdot\|$ on $\R^d$.  
Let $\bm_j\in\Z^d$ be given by $m_{j,i}=p_{j,i}q_{j+1}-p_{j+1,i}q_j$.  Then 
  $$\frac{\|\bm_j\|}{q_j(q_{j+1}+q_j)} \le \left\|\bx-\frac{\bp_j}{q_j}\right\| 
     \le \frac{2\|\bm_j\|}{q_jq_{j+1}}.$$  
\end{theorem}
This result, proved in \S\ref{S:best:approx}, essentially generalises the 
  fundamental inequalities satisfied by convergents of continued fractions.  

The basic idea behind the proof of Theorem~\ref{thm:main} is to cover 
  $\Sing(d)$ by sets of the form 
  $$\Delta(v)=\{\bx : v \text{ is a best approximation to } \bx\}.$$
The diameters of these sets are approximately (see (\ref{cor:diam})) 
  $$\frac{\delta_v}{|v|^{1+1/d}}$$ 
  where $\delta_v$ measures the distortion of some $d$-dimensional 
  lattice $\cL(v)$ associated to $v$.  
It turns out that these lattices become very distorted for best 
  approximations to points in $\Sing(d)$ (Theorem~\ref{thm:char}) 
  and this gives additional control over the size of $\Delta(v)$.  

To get Hausdorff dimension estimates, we use a technical device 
  introduced in \cite{Ch07} (self-similar coverings) which allows 
  one to systematically make refinements to the initial covering.  
Each refinement corresponds to a method of choosing subsequences 
  of best approximations that lead to more efficient covers.  
As is well-known, there can be arbitrarily long (consecutive) 
  sequences of ``colinear'' best approximations, i.e. lying on a 
  common (rational) line in $\R^d$.  
(See \cite{JL2}.)  The first ``method'', or \emph{acceleration} 
  as we prefer to call it, is to take the subsequence of best 
  approximations consisting of those that do not lie on the 
  rational line determined by the previous two.  
This almost gives the upper bound (Proposition~\ref{prop:key}) 
  except, for technical reasons, it is necessary to pass to a 
  further subsequence where the distortion parameter $\delta_v$ 
  is monotone.  
For the lower bound we construct Cantor-like subsets of $\Sing(d)$ 
  and use a lower bound estimate (Theorem~\ref{thm:gen:lower:1}) 
  involving a ``spacing condition'' (iii) in terms of a ``local 
  Hausdorff dimension formular'' (iv).  
As an illustration of the basic technique, we show that 
  the set of real numbers with divergent partial quotients is a 
  set of Hausdorff dimension $\tfrac12$ (Theorem~\ref{thm:dpq}), 
  although accelerations are not needed for this calculation.  

Perhaps the most significant contribution this paper makes to 
  the theory of simultaneous Diophantine approximation is the 
  introduction of the sets $\cL(v)$, which are computable objects 
  containing useful information about the approximation properties 
  of the rational point.  
It may be helpful to think of it as an object that encodes the 
  relative positions of other rationals in the vicinity of $v$.  
One of the main obstructions when addressing the case $d\ge3$ 
  is the fact that the moduli space for the lattice $\cL(v)$ is 
  no longer a space of rank one.  

\subsection*{Acknowledgments}
The author would like to thank Alex Eskin and Barak Weiss for 
  many discussions while this work was in progress.  
He would also like to thank the referee for a very careful 
  reading of an earlier version of this paper as well as the 
  many helpful suggestions improving the exposition.  
Last, but not least, the author would like to thank his wife, 
  Ying Xu, for her constant encouragement and unwavering support.

\section{Sequence of Best Approximations}\label{S:best:approx}
Let $\|\cdot\|$ be any norm on $\R^d$ and let $\|\cdot\|'$ denote the norm on 
  $\R^{d+1}$ given by $\|(\bx,y)\|':=\max(\|\bx\|,|y|)$.
For any $\bx\in\R^d$, let $W_\bx:\R\to\R$ be the function defined by 
  $$W_\bx(t)=\log\ell(g_th_\bx\Z^{d+1})$$ 
  where $\ell(\cdot)$ denotes the $\|\cdot\|'$-length of the shortest nonzero vector.  
\begin{lemma}\label{lem:PL}
The function $W_\bx$ is continuous, piecewise linear with slopes $-d$ and $+1$; 
  moreover, it has infinitely many local minima if and only if $\bx\not\in\Q^d$.  
\end{lemma}
\begin{proof}
For any $v\neq0$ the function $t\to\log\|g_tv\|'$ is a continuous, piecewise 
  linear function with at most one critical point.  
Its derivative is defined everywhere except at the critical point and is either 
  equal to $-d$ or $+1$.  
For each $\tau\in\R$, there is a finite set $F_\tau\subset h_\bx\Z^{d+1}$ such 
  that $W_\bx(t)=\log\ell(g_tF_\tau)$ for all $t$ in some neighborhood of $\tau$.  
Thus, $W_\bx$ is continuous, piecewise linear with slopes $-d$ and $+1$, because 
  it satisfies the same property locally.  

Let $C$ be the set of critical points of $W_\bx$ and note that $F_\tau$ can be 
  chosen so that it is constant on each connected component of $\R\setminus C$.  
If $\bx\not\in\Q^d$ then $\ell(g_tF_\tau)\to\infty$ as $t\to\infty$ whereas 
  Minkowski's theorem implies $\ell(g_th_\bx\Z^{d+1})$ is bounded above for all $t$.  
Hence, $F_t\neq F_\tau$ for some $t>\tau$ so that $C\cap[\tau,t]\neq\zip$ 
  and since $\tau$ can be chosen arbitrarily large, it follows that $W_\bx$ 
  has infinitely many local minima.  
If $\bx\in\Q^d$ then $\bx=\frac{\bp}{q}$ for some $\bp\in\Z^d,q\in\Z$ 
  such that $v=(\bp,q)$ satisfies $\gcd(v)=1$.  
Observe that we can take $F_\tau=\{h_\bx v\}$ for all sufficiently large $\tau$.  
Thus, $C$ is bounded, hence finite, and in particular $W_\bx$ has at most 
  finitely many local minima.  
\end{proof}

\begin{definition}
Let $v$ be a vector in $$Q:=\{(\bp,q)\in\Z^{d+1}:\gcd(\bp,q)=1,q>0\}$$ and 
  $\tau$ a local minimum time of the function $W_\bx$.  
We shall say ``$v$ realises the local minimum of $W_\bx$ at time $\tau$'' if 
  $W_\bx(t)=\log\|g_th_\bx v\|'$ for all $t$ in some neighborhood of $\tau$.  
The set of vectors in $Q$ that realise \emph{some} local minimum of $W_\bx$ 
  will be denoted by $$\Sigma(\bx).$$  
By convention, $\tau=+\infty$ is considered a local minimum time of $W_\bx$.  
\end{definition}

\begin{notation}
Given $\bx\in\R^d$ and $v=(\bp,q)\in Q$ we let 
\begin{equation}\label{def:hor:|.|}
    \hor_\bx(v):=\|q\bx-\bp\|\qquad\text{and}\qquad |v|:=q.  
\end{equation}
\end{notation}

\begin{lemma}\label{lem:realise}
Let $v\in Q$.  Then $v\in\Sigma(\bx)$ if and only if for any $u\in Q$ 
\begin{enumerate}
  \item[(i)]  $|u|<|v|$ implies $\hor_\bx(u)>\hor_\bx(v)$, and 
  \item[(ii)] $|u|=|v|$ implies $\hor_\bx(u)\ge\hor_\bx(v)$.  
\end{enumerate}
\end{lemma}
\begin{proof}
Suppose $v\in\Sigma(\bx)$, so that it realises a local minimum of $W_\bx$.  
If $u\in Q$ does not satisfy (i) then $\|g_th_\bx u\|'\le\|g_th_\bx v\|'$ 
  for all $t$, with strict inequality for at least some $t$, which implies 
  that $v$ cannot realise a local minimum of $W_\bx$, a contradiction.  
Therefore, (i) holds for any $u\in Q$.  The argument that (ii) holds 
  as well is similar.  
Conversely, suppose (i) and (ii) hold for all $u\in Q$.  
Let $(\tau,\eps)$ be the unique local minimum of $t\to\|g_th_\bx v\|'$.  
Let $B'$ be the closed $\|\cdot\|'$-ball of radius $\eps$ at the origin.  
Write it as $B\times I$ where $B\subset\R^d$ and $I\subset\R$.  
Let $Z=B'\cap g_th_\bx\Z^{d+1}$ and $Z^*=Z\setminus\{0\}$.  
Then (i) and (ii) imply $Z^*\subset\partial B\times\partial I$, which 
  implies $v$ realises the shortest nonzero vector at time $\tau$.  
(By this we mean $g_\tau h_\bx v$ is the shortest nonzero vector in 
  $g_\tau h_\bx\Z^{d+1}$.)  
Since there is exists a slightly larger ball $B''$ containing $B'$ 
  such that $B''\cap g_th_\bx\Z^{d+1}=Z$, it follows that $v$ realises 
  the shortest nonzero vector for an interval of $t$ about $\tau$.  
Thus, $v$ realises a local minimum of $W_\bx$ and $v\in\Sigma(\bx)$.  
\end{proof}

The next lemma was proved in \cite{JL2} for the case $d=2$.  
\begin{lemma}\label{lem:consec}
If $u,v\in\Sigma(\bx)$ realise a consecutive pair of local minima of 
  the function $W_\bx$, then they span a \emph{primitive} two-dimensional 
  sublattice of $\Z^{d+1}$, i.e. $\Z^{d+1}\cap(\R u + \R v)=\Z u + \Z v$.  
\end{lemma}
\begin{proof}
Let $F=\{\pm u,\pm v\}$ and denote its convex hull by $\conv(F)$.  
Let $\cC(\bx)$ be the collection of subsets of $\R^{d+1}$ of the form 
  $$C(r,h)=\{(\ba,b):\|\ba\|\le r,|b|\le h\} \qquad r>0,\;h>0$$ that 
  intersect $h_\bx(\Z^{d+1}\setminus\{0\})$ on the boundary but not in 
  the interior.  
Observe that the set of maximal (resp. minimal) elements of $\cC(\bx)$, 
  partially ordered by inclusion, is in one-to-one correspondence with 
  the set of local maxima (resp. local minima) of the function $W_\bx$.  
Since $u$ and $v$ realise distinct local minima, $|u|\neq|v|$.  
Hence, without loss of generality, assume that $|u|<|v|$.  
The element of $\cC(\bx)$ corresponding to the unique local maximum of 
  $W_\bx$ between the consecutive pair of local minima determined by $u$ 
  and $v$ is given by the parameters $r=\hor_\bx(u)$ and $h=|v|$.  
Note that $\conv(h_\bx F)$ is a subset of $C(r,h)$ and intersects the 
  boundary of $C(r,h)$ in the four points of $h_\bx F$.  
Therefore, $\conv(F)\cap\Z^{d+1}=F\cup\{0\}$, which is equivalent to 
  $u,v$ spanning a primitive two-dimensional sublattice of $\Z^{d+1}$.  
\end{proof}

Recall that $\frac{\bp}{q}\in\Q^d$ is a \emph{best approximation} to $\bx$ if 
\begin{enumerate}
  \item[(i)]  $\|q\bx-\bp\| < \|n\bx-\bm\|$ for any $(\bm,n)\in\Z^{d+1}$, $0<n<q$, 
  \item[(ii)] $\|q\bx-\bp\|\le\|q\bx-\bp'\|$ for any $\bp'\in\Z^d$.  
\end{enumerate}
The study of best approximations is central to the theory of simultaneous 
  Diophantine approximation and has a long history going back to Lagrange, 
  who showed that the sequence of best approximations in the case $d=1$ 
  are enumerated by the convergents in the continued fraction expansion.  
The literature on this subject is extensive.  We refer the reader to the 
  papers \cite{JL1} and \cite{JL3}, which contain many further references.  

Lemma~\ref{lem:realise} gives a simple dynamical interpretation for 
  the sequence of best approximations, ordered by increasing height: 
  \emph{they correspond precisely to the sequence of vectors that 
  realise the local minima of $W_\bx$}.  
\begin{notation}
For any $v\in Q$ let 
\begin{equation}\label{def:dot:v}
    \Dot v:=\frac{\bp}{q}\in\Q \quad\text{where}\quad v=(\bp,q)\in Q.  
\end{equation}
\end{notation}

We shall often ignore the distinction between a vector $v\in Q$ and 
  the rational $\Dot v$ corresponding to it.  
Thus, we may refer to a sequence $(v_j)$ in $Q$ as the sequence of 
  best approximations to $\bx$, by which we mean for every $j$ the 
  vector $v_j$ realises the $j$th local minimum of $W_\bx$.  
This raises the issue of the uniqueness of the vector realising a 
  local minimum of $W_\bx$, or equivalently, the existence of best 
  approximations to $\bx$ of the same height.  
In the case $d=1$, this can only happen if $\bx$ is a half integer.  
In general, there are at most finitely many local minima that can 
  be realised by multiple vectors in $Q$.  
(See the remark following Theorem~\ref{thm:diam} below.)  
When referring to ``the sequence of best approximations to $\bx$'' 
  we really mean \emph{any} sequence $(v_j)$ such that the $j$th 
  local minimum of $W_\bx$ is realised by $v_j$.

\subsection{Two-dimensional sublattices}
\begin{definition}\label{def:script:L(v)}
For any $v\in Q$, we denote by $$\cL(v)$$ the set of primitive 
  two-dimensional sublattices of $\Z^{d+1}$ containing $v$.  
\end{definition}

There is a natural way to view $\cL(v)$ as the set of primitive 
  elements in some $d$-dimensional lattice.  
Indeed, consider the exact sequence of real vector spaces 
\begin{equation}\label{exact:seq}
  \R \longrightarrow \Wedge^1\R^{d+1} \stackrel{\varphi}{\longrightarrow} 
     \Wedge^2\R^{d+1} \longrightarrow \Wedge^3\R^{d+1}
\end{equation}
  where each map is given by exterior multiplication by $v$.  
Since $v\neq0$, the kernel of $\varphi$ is one-dimensional, 
  from which it follows that the image of $\varphi$, denoted 
  $$\cL_\R(v),$$ is a real vector space of dimension $d$.  
Similarly, we have an exact sequence of free $\Z$-modules 
\begin{equation*}
  \Z \longrightarrow \Wedge^1\Z^{d+1} \longrightarrow 
     \Wedge^2\Z^{d+1} \longrightarrow \Wedge^3\Z^{d+1}
\end{equation*}
  where the image of second map is a free $\Z$-module of rank $d$, 
  denoted $$\cL_\Z(v).$$  
It embeds $\cL_\Z(v)$ as a $d$-dimensional lattice in $\cL_\R(v)$.  
The set of \emph{oriented}, primitive, two-dimensional sublattices 
  of $\Z^{d+1}$ that contain $v$ is given by 
    $$\cL_+(v) = \{u\wedge v : u,v\in Q, \Z u+\Z v\in\cL(v)\}.$$  
The next lemma shows that $\cL_+(v)$ coincides with the set of 
  primitive elements of lattice $\cL_\Z(v)$.  

\begin{lemma}\label{lem:primitive}
Let $L=\Z u+\Z v$ be a two dimensional sublattice of $\Z^{d+1}$.  
Then $L$ is primitive if and only if $u\wedge v$ is primitive as an 
  element of $\cL_\Z(v)$, i.e. $u\wedge v\neq dw$ for any $d\ge2$ 
  and $w\in\cL_\Z(v)$.  
\end{lemma}
\begin{proof}
Let $L'=\Z^{d+1}\cap\Span L$ so that $L$ is primitive iff $L'=L$.  
Suppose $u\wedge v=dw$ for some $d\ge2$ and $w\in\cL_\Z(v)$.  
Then $w=u'\wedge v$ for some $u'\in\Z^{d+1}$.  
Since $(du'-u)\wedge v=0$, we have $du'=u+cv$ for some $c\in\Z$.  
Since $d\ge2$, $u'\notin L$.  Hence, $L'\neq L$.  
Conversely, suppose $L'\neq L$.  
Choose $u'\in Q$ so that $L'=\Z u'+\Z v$.  
Since $L\subset L'$, we may write $u=au'+bv$ for some $a,b\in\Z$.  
Then $u\wedge v=au'\wedge v$.  Since the index of $L$ in $L'$ is 
  given by $|a|$, we have $u\wedge v=dw$ where $d=|a|$ and 
  $w=\pm u'\wedge v\in\cL_\Z(v)$.  
Since $L\neq L'$, $d\ge2$.  
\end{proof}

Identifying $\Wedge^1\R^{d+1}$ with $\R^{d+1}$, we note that the kernel 
  of $\varphi$ is given by the one-dimensional subspace $\R v$.  
Thus, $\varphi$ induces an isomorphism of $\cL_\R(v)$ with the space 
  of cosets of $\R v$ in $\R^{d+1}$.  The elements in $\cL_\Z(v)$ 
  correspond to cosets that have nonempty intersection with $\Z^{d+1}$.  
Let $E_+$ be the expanding eigenspace for the action of $g_1$.  
Then $$\R^{d+1}=E_+\oplus\R v$$ and the map that sends a coset of $\R v$ 
  to the point of intersection with $E_+$ induces an isomorphism of 
  of $\cL_\R(v)$ with $E_+$.  
The norm $\|\cdot\|$ on $\R^d$, which is naturally identified with 
  $E_+$, induces a norm on $\cL_\R(v)$, which we shall denote by 
  $$\|\cdot\|_{\cL(v)}.$$  

Let $E_-$ be the contracting eigenspace for $g_1$.  Since $\dim E_-=1$, 
  the $k$th exterior power decomposes into two eigenspaces for $g_t$ 
  $$\Wedge^k\R^{d+1} = E_+^k \oplus E_-^k$$ where $E_+^k$ and $E_-^k$ 
  are naturally identified with $\Wedge^kE_+$ and $\Wedge^{k-1}E_+$, 
  respectively.  
Let $e_1,\dots,e_{d+1}$ be the standard basis vectors for $\R^{d+1}$.  
The operation of wedging with $e_{d+1}$ induces an isomorphism between 
  $E_+^k$ and $E_-^{k+1}$; in particular, we have an isomorphism of 
  $E_-^2$ with $E_+^1$, which is naturally identified with $\R^d$ 
  through the isomorphisms with $E_+$.  
The norm $\|\cdot\|$ on $\R^d$ induces a norm on $E_-^2$, which may be 
  extended to a seminorm on all of $\Wedge^2\R^{d+1}$ by defining the 
  (semi)norm of an element to be the norm of the component in $E_-^2$.  
This seminorm will be denoted by $$|\cdot|.$$  

Given $L\in\cL(v)$ we may form an element $u\wedge v\in\cL_\Z(v)$ 
  by choosing any pair in $Q$ such that $L=\Z u+\Z v$.  
This element is well-defined up to sign, so it makes sense to talk 
  about the norm of $L$ an element in $\cL_\R(v)$ and also as an 
  element of $\Wedge^2\R^{d+1}$; denote these, respectively, by 
  $$\|L\|_{\cL(v)}\quad\text{and}\quad |L|.$$  
There is a simple relation between these norms.  
Note that the action of $h_\bx$ on an element in $\Wedge^2\R^{d+1}$ 
  preserves the component in $E^2_-$.  
Now let $u',v'$ be the respective images of $u,v$ under $h_{\Dot v}$.  
Then $u'\wedge v' = |v| u'_+\wedge e_{d+1}$ where $u'_+$ is the 
  component of $u'$ in $E_+$.  
Note that $u'_+$ is precisely the point where the coset of $\R v$ 
  corresponding to $L$ intersects $E_+$.  
Note also that its norm is given by $\hor_{\Dot v}(u)$.  
It follows that 
\begin{equation}\label{eq:norms}
  \|L\|_{\cL(v)} = \hor_{\Dot v}(u) = \frac{|L|}{|v|}.  
\end{equation}

The image of $\cL_\Z(v)$ under the isomorphism of $\cL_\R(v)$ 
  with $E_+$ is simply the image of $\Z^{d+1}$ under the 
  projection of $\R^{d+1}=E_+\oplus\R v$ onto $E_+$, i.e.  
  the projection along lines parallel to $v$.  
Alternatively, it can be described as the set of all 
  components in $E_+$ of vectors in $h_{\Dot v}\Z^{d+1}$.  
Its volume is given by 
  $$\vol\big(\cL_\Z(v)\big) = \frac{1}{|v|}.$$  

Since $\cL_+(v)$ is a discrete subset of a normed vector space, 
  there exists an element of minimal positive norm.  
While this element may not be unique, we shall choose one for 
  each $v\in Q$, once and for all, and denote it by $$L(v).$$  
The corresponding element in $\cL(v)$ will be denoted by the 
  same symbol.\footnote{We shall often blur the distinction 
  between elements in $\cL_+(v)$ and $\cL(v)$, leaving it to 
  the context to determine which meaning is intended.}  
Since the norm of the shortest nonzero vector in a unimodular 
  lattice in $\R^d$ is bounded above by some constant $\mu_0$ 
  (depending on $\|\cdot\|$) we have for all $v\in Q$ 
\begin{equation}\label{def:mu0}
  \frac{|L(v)|}{|v|^{1-1/d}} \le \mu_0.  
\end{equation}

Let us mention a form of (\ref{eq:norms}) that is symmetric 
  with respect to $u$ and $v$.  
Let $\dist(\cdot,\cdot)$ denote the metric on $\R^d$ induced 
  by the norm $\|\cdot\|$.  
Then $\hor_{\Dot v}(u) = |u|\dist(\Dot u,\Dot v)$ so that 
\begin{equation}\label{eq:dist}
  \dist(\Dot u,\Dot v)=\frac{|u\wedge v|}{|u||v|}.  
\end{equation}
Let us extend the notation $\Dot v,|v|$ and $\hor_\bx(v)$ 
  introduced in (\ref{def:dot:v}) and (\ref{def:hor:|.|}) 
  to the set $E_+^c=\R^{d+1}\setminus E_+$.  
Then (\ref{eq:dist}) holds for all $u,v\in E_+^c$.  
It will be convenient to allow overscripts on the arguments of 
  $\dist(\cdot,\cdot)$ to be dropped; formally, we are extending 
  $\dist:\R^d\times\R^d\to\R$ to a function that is defined on 
  $E\times E$ where $E$ is the disjoint union of $\R^d$ and $E_+^c$.  
These conventions allow for more appealing formulas such as 
  $$|u\wedge v|=|u||v|\dist(u,v),\quad \hor_\bx(v)=|v|\dist(v,\bx).$$

\subsection{Domains of approximation}
We now investigate the sets $$\Delta(v):=\{\bx\in\R^d:v\in\Sigma(\bx)\}$$ 
  for $v\in Q$.  
By Lemma~\ref{lem:realise}, $$\Delta(v) = \big(\cap_{|u|<|v|} \Delta_u(v)\big)
      \cap \big(\cap_{|u|=|v|}\overline{\Delta_u(v)}\big)$$ 
  where the sets $\Delta_u(v)$, defined only for $u\in Q\setminus\{v\}$, 
  are given by $$\Delta_u(v)=\{\bx\in\R^d: \hor_\bx(u)>\hor_\bx(v) \}.$$  
We note that $\Delta_u(v)$ is bounded iff $|u|<|v|$.  

\begin{lemma}\label{lem:dist:w}
For any (distinct) $u,v\in Q$ with $|u|\le|v|$ we have 
  $$\dist(\bx,u) > \dist(u+v,u)\quad\forall\bx\in\Delta_u(v).$$  
Here, $\dist(u+v,\cdot)$ means $\dist(\Dot w,\cdot)$ where $w=u+v$.  
\end{lemma}
\begin{proof}
By definition, $\bx\in\Delta_u(v)$ iff $$\dist(v,\bx)<\lambda\dist(u,\bx)
   \quad\text{where}\quad \lambda=\frac{|u|}{|v|}\le1.$$  
Let $w=u+v$.  Since $\dist(u,v)=\dist(u,w) + \dist(w,v)$ and 
  $$\dist(v,w) = \frac{|v\wedge w|}{|v||w|} = \frac{|u|}{|v|} 
       \left(\frac{|u\wedge w|}{|u||w|}\right) = \lambda\dist(u,w)$$ 
  the triangle inequality implies 
  $$(1+\lambda)\dist(u,w)\le\dist(u,\bx)+\dist(v,\bx)<(1+\lambda)\dist(u,\bx)$$
  and the lemma follows.  
\end{proof}

\begin{remark}
It follows easily from Lemma~\ref{lem:dist:w} that the infimum of 
  $$\{\dist(u,\bx):\bx\in\overline{\Delta_u(v)}\}$$ is realised at the 
  rational point corresponding to $u+v$.  It can similarly be shown that 
  the supremum is realised by the rational corresponding to $v-u$.  
In the case when $\|\cdot\|$ is the Euclidean norm, the set $\Delta_u(v)$ 
  is the open Euclidean ball having these points as antipodal points.  
\end{remark}

\begin{lemma}\label{lem:u+-}
Given $L\in\cL_+(v)$ there are unique vectors $u_\pm\in Q$ satisfying 
  $L=u_\pm\wedge v$ and $|u_\pm|\le|v|$.  
Moreover, $|u_+|<|v|$ iff $v=u_++u_-$ iff $|u_-|<|v|$.  
Similarly, $|u_+|=|v|$ iff $2v=u_++u_-$ iff $|u_-|=|v|$.  
\end{lemma}
\begin{proof}
Existence and unique of $u_\pm$ is clear.  If $|u_+|<|v|$ then $v-u_+$ 
  satisfies the conditions defining $u_-$ so that $v-u_+=u_-$.  
Similarly, if $|u_+|=|v|$ then $2v-u_+=u_-$.  
\end{proof}

Let $B(\bx,r)\subset\R^d$ denote the open ball at $\bx$ of radius $r$.  
\begin{theorem}\label{thm:diam}
Given $v\in Q$ with $|v|>1$, let $r=\tfrac{|L(v)|}{|v|^2}$.  Then 
\begin{equation}\label{ieq:diam}
    B(\Dot v,\frac{r}{2})\subset\Delta(v)\subset B(\Dot v,2r).  
\end{equation}
\end{theorem}
\begin{proof}
Let $L\in\cL(v)$ and choose the orientation for it such that 
  $u_\pm$ in Lemma~\ref{lem:u+-} satisfy $|u_+|\ge|u_-|$.  
For any $\bx\in\Delta_{u_-}(v)$ we have 
  $$\dist(\bx,v)<\frac{|u_-|}{|v|}\dist(\bx,u_-)<\frac{|u_-|}{|v|}
     \left(\dist(\bx,v) + \frac{|L|}{|u_-||v|}\right)$$
  so that 
  $$\left(1-\frac{|u_-|}{|v|}\right)\dist(\bx,v)<\frac{|L|}{|v|^2}.$$  
If $|u_+|<|v|$ then it follows that 
\begin{equation}\label{ieq:diam:u+}
  \dist(\bx,v) < \frac{|L|}{|u_+||v|} \le \frac{2|L|}{|v|^2}.  
\end{equation}
We claim $L$ could have been chosen initially so that $|u_+|<|v|$.  
Indeed, let $B'(r')=\overline{B(0,r')}\times(-|v|,|v|)$ and 
  note that since $|v|>1$ there is a smallest $r'>0$ such that 
  $B'(r')\cap h_{\Dot v}\Z^{d+1} \not\subset E_+.$  
Then $Q\cap h_{-\Dot v}B'(r')\neq\zip$ and the desired property 
  for $L$ is satisfied by $\Z u+\Z v$ for any $u$ in this set.  
This proves the claim, and hence $\Delta(v)\subset B(\Dot v,2r)$.  

For the other inclusion, consider $u\in Q$ with $|u|\le|v|$ and $u\neq v$.  
Let $\lambda=\tfrac{|u|}{|v|}$ and $\mu>0$.  
For any $\bx\in B(\Dot v,\mu r)$ we have 
  $$\dist(\bx,v) < \frac{\mu|L|}{|v|^2} \le \lambda\mu\frac{|u\wedge v|}{|u||v|}$$ 
  whereas 
  $$\dist(\bx,u)>(1-\lambda\mu)\frac{|u\wedge v|}{|u||v|}.$$  
Thus, $\bx\in\Delta_u(v)$ provided $$\mu\le1-\lambda\mu$$ 
  and since $\lambda\le1$, this holds for $\mu=\tfrac12$.  
\end{proof}

As a corollary of Theorem~\ref{thm:diam} we get 
\begin{equation}\label{cor:diam}
  \diam\Delta(v)\asymp\frac{|L(v)|}{|v|^2}\le\frac{\mu_0}{|v|^{1+1/d}}. 
\end{equation}
  where $\mu_0$ is the constant satisfying (\ref{def:mu0}).  
Here, $A\asymp B$ means $C^{-1}B\le A\le CB$ for some constant $C$.  

\begin{remark}
Observe that if $u,v\in Q$ are distinct vectors that realise 
  the same local minimum of $W_\bx$, then $|u|=|v|$ so that we 
  can think of $u-v$ as an element of $\Z^d$; moreover, 
  $$\|u-v\|=|v|\dist(u,v)\le\frac{8\mu_0}{|v|^{1/d}}.$$  
If we let $\lambda_0$ be the norm of the shortest nonzero vector 
  in $\Z^d$, then it follows that any vector $v\in\Sigma(\bx)$ 
  satisfying $$|v|>\frac{8^d\mu_0^d}{\lambda_0^d}$$ is uniquely 
  determined by the local minimum that it realises.  
In other words, the sequence of best approximations is eventually 
  uniquely determined.  
This fact had already been observed in \cite{JL1}.  
\end{remark}

\begin{theorem}\label{thm:best:1}
Let $v\in\Sigma(\bx)$ and suppose $u\in Q$ is such that $|u|<|v|$.  Then 
\begin{equation}\label{ieq:best:1}
  \frac{1}{2}\dist(u,v) < \dist(u,\bx) < 2\dist(u,v).  
\end{equation}
\end{theorem}
\begin{proof}
Let $L\in\cL(v)$ be the one containing $u$.  
Then $|u\wedge v|=b|L|$ for some positive integer $b$.  
By Lemma~\ref{lem:dist:w} we have 
  $$\dist(\bx,u) > \dist(u+v,u) = \frac{b|L|}{|u+v||u|} 
       \ge \frac{b|L|}{2|u||v|} = \frac12\dist(u,v).$$  
From the \emph{first} inequality in (\ref{ieq:diam:u+}) we have 
\begin{align*}
  \dist(\bx,u) &\le \dist(u,v) + \dist(v,\bx) \\
               &< \frac{b|L|}{|u||v|} + \frac{|L|}{|u_+||v|} 
                = \left(1+\frac{|u|}{b|u_+|}\right)\dist(u,v).  
\end{align*}
If $b\ge2$ the expression in parentheses is at most $2$.  
If $b=1$ then $u=u_\pm$ and the same is true again.  
Thus, $\dist(\bx,u)<2\dist(u,v)$.  
\end{proof}

Theorem~\ref{thm:main:4} is a consequence of the following.  
\begin{theorem}\label{thm:best:2}
Let $\tfrac{\bp_j}{q_j}$ be the sequence of best approximations to $\bx$ 
  and set $v_j=(\bp_j,q_j)$ and $L_{j+1}=\Z v_{j+1}+\Z v_j$.  Then 
\begin{equation}\label{ieq:best:2}
  \frac{|L_{j+1}|}{q_j(q_{j+1}+q_j)} \le \left\|\bx-\frac{\bp_j}{q_j}\right\| 
     \le \frac{2|L_{j+1}|}{q_jq_{j+1}}.  
\end{equation}
\end{theorem}
\begin{proof}
The first (resp. second) inequality follows by applying Lemma~\ref{lem:dist:w} 
  (resp. Theorem~\ref{thm:best:1}) with $u=v_j$ and $v=v_{j+1}$.  
\end{proof}

It can be shown that $|L_j|\le C|L(v_j)|$ for all $j$ (and $\bx$) 
  where $C$ is a constant that depends only on the norm $\|\cdot\|$.  
Using this, we may rewrite (\ref{ieq:best:2}) as 
  $$\|q_j\bx-\bp_j\| \asymp \frac{\delta(v_{j+1})}{q_{j+1}^{1/d}}$$ 
  where 
\begin{equation}\label{def:delta(v)}
  \delta(v) = |v|^{1/d}\|L(v)\|_{\cL(v)} \le \mu_0.  
\end{equation}
As we shall see $\bx\in\Sing(d)$ if and only if 
  $\delta(v_j)\to0$ as $j\to\infty$.  
See Theorem~\ref{thm:char} below.

\subsection{Characterisation of singular vectors}
The sequence of critical times of the function $W_\bx$ are ordered by 
  $$\tau_0 < t_0 < \tau_1 < t_1 < \dots $$ where 
  $\tau_j$ (resp. $t_j$) is the $j$th local maximum (resp. minimum) time.  
Note that the first critical point $\tau_0$ is a local maximum because 
  as $t\to-\infty$ we have $W_\bx(t)=t+\log\|v_{-1}\|$ where $v_{-1}$ is 
  any nonzero vector in $\Z^d$ of minimal $\|\cdot\|$-norm.  

\begin{definition}\label{def:loc:max}
For $u,v\in Q$ and $\bx\in\R^d$, let $\eps_\bx(u,v)$ be defined by 
  $$\eps_\bx(u,v)^{1+1/d}=|v|^{1/d}\hor_\bx(u).$$  
\end{definition}

\begin{lemma}\label{lem:common}
For any $u,v\in Q$ with $|u|<|v|$ and for any $\bx\in\Delta(v)$ there is 
  a unique time $\tau$ when $\|g_\tau h_\bx u\|'=\|g_\tau h_\bx v\|'$.  
Moreover, the common length is given by $\eps_\bx(u,v)$.  
\end{lemma}
\begin{proof}
Since $|u|<|v|$ and $\bx\in\Delta(v)$ we have $\hor_\bx(u)>\hor_\bx(v)$.  
This implies the existence of $\tau$.  Let $\eps=\eps_\bx(u,v)$.  
The point $(\tau,\log\eps)$ is where the lines $y=-d(t-a)$ and $y=(t-b)$ 
  meet, where $a=\frac{1}{d}\log|v|$ and $b=-\log\hor_\bx(u)$.  
Since $$(t,y)=\left(\frac{ad+b}{d+1},\frac{ad-bd}{d+1}\right)$$ we have 
\begin{align*}
  \log\eps &= \frac{1}{d+1}\log|v|+\frac{d}{d+1}\log\hor_\bx(u), \\
      \tau &= \frac{1}{d+1}\log|v|-\frac{1}{d+1}\log\hor_\bx(u) 
\end{align*}
  from which it follows that 
  $$\eps^{1+1/d}=|v|^{1/d}\hor_\bx(u) \quad\text{and}\quad 
    e^{-(d+1)\tau} = \frac{\hor_\bx(u)}{|v|}.$$  
\end{proof}

\begin{lemma}\label{lem:char}
Assume $\bx\not\in\Q^d$ and set $\delta_j=\eps_\bx(v_{j-1},v_j)^{1+1/d}$ 
  where $(v_j)$ is the sequence of best approximations to $\bx$.  
Then $\bx\in\DI_\delta(d)$ iff $\delta_j<\delta$ for all sufficiently large $j$.  
\end{lemma}
\begin{proof}
Write $v_j=(\bp_j,q_j)$ so that Lemma~\ref{lem:common} implies 
  $$\|q_j\bx-\bp_j\|=\hor_\bx(v_j)=\frac{\delta_{j+1}}{q_{j+1}^{1/d}}.$$  
If $\delta_j<\delta$ then $(\bp_j,q_j)$ solves (\ref{ieq:sing}) 
  for all $q_j< T\le q_{j+1}$.  
It follows that $\delta_j<\delta$ for all large enough $j$ implies 
  $\bx\in\DI_\delta(d)$.  
Conversely, suppose $\bx\in\DI_\delta(d)$ so that there exists $T_0$ such 
  that (\ref{ieq:sing}) admits a solution for all $T>T_0$.    
Suppose $j$ is large enough so that $q_{j+1}>T_0$.  
Let $(\bp,q)$ be a solution to (\ref{ieq:sing}) for $T=q_{j+1}$.  
Then $q<q_{j+1}$, implying that 
  $$\|q_j\bx-\bp_j\|\le\|q\bx-\bp\|<\frac{\delta}{q_{j+1}^{1/d}}$$ 
  from which it follows that $\delta_j<\delta$.  
\end{proof}

The next theorem gives a characterisation is purely in terms of the 
  sequence of best approximations, without explicit reference to $\bx$.  
This will be needed to reduce the main task to a problem in symbolic dynamics.  
\begin{theorem}\label{thm:char}
Assume $\bx\notin\Q^d$ and for each $j\ge1$ set 
  $$\eps_j^{1+1/d}=\frac{|v_{j-1}\wedge v_j|}{|v_j|^{1-1/d}}$$ 
  where $(v_j)_{j\ge0}$ is the sequence of best approximations to $\bx$.  
Assume $\delta>0$ and $\eps>0$ be related by $\delta=\eps^{1+1/d}$.  
Then $\bx\in\DI_{\delta/2}(d)$ implies $\eps_j<\eps$ for all sufficiently 
  large $j$, which in turn implies $\bx\in\DI_{2\delta}(d)$.  
In particular, $\bx\in\Sing(d)$ iff $\eps_j\to0$ as $j\to\infty$.  
\end{theorem}
\begin{proof}
By Lemmas~\ref{lem:char} is suffices to show 
  $$\frac12\eps_j^{1+1/d}\le\eps_\bx(v_{j-1},v_j)\le2\eps_j^{1+1/d}$$
  which holds by Theorem~\ref{thm:best:1}.  
\end{proof}

\section{Self-similar coverings}\label{S:self-sim}
Let $X$ be a metric space and $J$ a countable set.  
Given $\sigma\subset J\times J$ and $\alpha\in J$ we let $\sigma(\alpha)$ 
  denote the set of all $\alpha'\in J$ such that $(\alpha,\alpha')\in\sigma$.  
We say a sequence $(\alpha_j)$ of elements in $J$ is $\sigma$-\emph{admissible} 
  if $\alpha_{j+1}\in\sigma(\alpha_j)$ for all $j$, and let $J^\sigma$ denote 
  the set of all $\sigma$-admissible sequences in $J$.  
By a \emph{self-similar covering} of $X$ we mean a triple $(\cB,J,\sigma)$ 
  where $\cB$ is a collection of bounded subsets of $X$, $J$ a countable 
  index set for $\cB$, and $\sigma\subset J\times J$ such that there is a 
  map $\cE:X\to J^\sigma$ that assigns to each $x\in X$ a $\sigma$-admissible 
  sequence $(\alpha_j^\bx)$ such that for all $x\in X$ 
\begin{enumerate}
  \item[(i)] $\cap B(\alpha_j^\bx)=\{x\}$, and 
  \item[(ii)] $\diam B(\alpha_j^\bx)\to0$ as $j\to\infty$, 
\end{enumerate}
  where $B(\alpha)$ denotes the element of $\cB$ indexed by $\alpha$.  

\begin{theorem}\label{thm:gen:upper}
(\cite{Ch07}, Theorem~5.3)
If $X$ is a metric space that admits a self-similar covering $(\cB,J,\sigma)$, 
  then $\Hdim X \le s(\cB,J,\sigma)$ where 
\begin{equation}\label{def:s}
  s(\cB,J,\sigma)=\sup_{\alpha\in J}
      \inf \left\{s>0 :  \sum_{\alpha'\in\sigma(\alpha)} 
          \left(\frac{\diam B(\alpha')}{\diam B(\alpha)}\right)^s\le 1 \right\}.  
\end{equation}
\end{theorem}

In many applications, $X$ is a subset of some ambient metric space $Y$ and 
  we are given a self-similar covering $(\cB,J,\sigma)$ of $X$ by bounded 
  subsets of $Y$ rather than $X$.  
For any bounded subset $B\subset Y$ we have $\diam_X B\cap X\le \diam_Y B$ 
  but equality need not hold in general.  
To compute $s(cB,J,\sigma)$ one would also need an inequality going in the 
  other direction.  While such an inequality may not be difficult to obtain, 
  it is both awkward and unnecessary:  \emph{Theorem~\ref{thm:gen:upper} 
  remains valid in this more general situation if the diameters in 
  (\ref{def:s}) are taken with respect to the metric of $Y$.}  
The proof of this more general statement does not follow directly from 
  Theorem~\ref{thm:gen:upper}, but the argument given in \cite{Ch07} applies 
  with essentially no change and will not be repeated here.  

For lower bounds, we shall use 
\begin{theorem}\label{thm:gen:lower:1}
Let $\cB$ be a collection of nonempty compact subsets of a metric space $Y$ 
  indexed by a countable set $J$.  
Suppose $X\subset Y$ and $\sigma\subset J\times J$ are such that 
\begin{enumerate}
  \item[(i)] for each $\alpha\in J$, $\sigma(\alpha)$ is a finite subset of $J$ 
    with at least $2$ elements and for each $\alpha'\in\sigma(\alpha)$ we have 
    $B(\alpha')\subset B(\alpha)$, 
 \item[(ii)] for each $(\alpha_j)\in J^\sigma$, we have $\diam B(\alpha_j)\to0$ 
    and the unique point in $\cap B(\alpha_j)$ belongs to $X$, 
\item[(iii)] there exists $\rho>0$ such that for any $\alpha\in J$ and 
    for any distinct pair $\alpha',\alpha''\in\sigma(\alpha)$ 
\begin{equation}\label{ieq:spacing:1}
           \dist(B(\alpha'),B(\alpha''))>\rho\diam B(\alpha), 
\end{equation}
 \item[(iv)] there exists $s>0$ such that for every $\alpha\in J$, 
\begin{equation}\label{ieq:lower:1}
      \sum_{\alpha'\in\sigma(\alpha)} 
              [\diam B(\alpha')]^s \ge [\diam B(\alpha)]^s.  
\end{equation}
\end{enumerate}
Then $\Hdim X\ge s$.  
\end{theorem}

Before giving a proof of Theorem~\ref{thm:gen:lower:1}, let us discuss the 
  significance of the spacing condition (iii).  
There are many theorems in the literature that can be applied to give lower 
  bounds on $\Hdim X$ in the setup of Theorem~\ref{thm:gen:lower:1}.  
Different assumptions on the set $X$ lead to different lower bound estimates.  
Consider the following \emph{distorted} Cantor set $$C_\delta,\quad 0<\delta<1$$ 
  defined as the intersection $\cap_{k\ge0} E_k$ where $E_0=[0,1]$, 
  $E_1=[0,\tfrac{\delta}{2}]\cup[\tfrac12,1]$ and for each $k\ge2$, the set $E_k$ 
  is a disjoint union of $2^k$ closed intervals obtained by removing a similar 
  $\tfrac{1-\delta}{2}$ fraction from each of the intervals of $E_{k-1}$.  
For each interval $I$ of $E_k$ the density of the intervals of $E_{k+1}$ in $I$ 
   is a constant $$\Delta_k=\frac{1+\delta}{2}$$ independent of both $I$ and $k$.  
The length of each interval in $E_k$ is bounded between 
  $$d^-_k=\frac{\delta^k}{2^k} \quad\text{and}\quad d^+_k=\frac{1}{2^k}.$$  
The length of the smallest gap between the intervals of $E_k$ is 
  $$\eps_k = \frac{(1-\delta)d^-_{k-1}}{2}$$ and each interval of $E_{k-1}$ 
  has exactly $$m_k=2$$ intervals of $E_k$ contained in it.  
The estimate in \cite{Mc87} based on density gives a lower bound 
  $$h_d(\delta) 
     = 1-\limsup_{k\to\infty} \frac{\sum_{i=1}^{k+1}|\log\Delta_i|}{|\log d^+_k|} 
     = \frac{\log(1+\delta)}{\log 2}$$
  whereas the estimate in \cite{Fa03} (p.64) based on gaps gives 
  $$h_g(\delta) 
     = \liminf_{k\to\infty} \frac{\log(m_1\cdots m_{k-1})}{-\log m_k\eps_k} 
     = \frac{\log 2}{\log(2/\delta)}.$$  
Both functions increase from $0$ to $1$ as $\delta$ ranges from $0$ to $1$, 
  and since 
\begin{align*}
  h'_d(0+)&=\frac{1}{\log2}, & h'_d(1-)&=\frac{1}{2\log2}, \\
  h'_g(0+)&=\infty, & h'_g(1-&)=\frac{1}{\log2}, 
\end{align*}
  we have 
  $$\lim_{\delta\to0^+}\frac{h_d(\delta)}{h_g(\delta)}=0,\quad\text{and}\quad 
    \lim_{\delta\to1^-}\frac{h_d(\delta)}{h_g(\delta)}=2.$$  
Note that the lower bounds are comparable in the limit as $\delta$ approaches 
  one, but the one using a gap hypothesis is infinitely better in the limit as 
  $\delta$ approaches zero.\footnote{The graphs of $h_d$ and $h_g$ cross near 
  the point $(.2726604^+,.3478475^+)$.}  

The exact value $h(\delta)$ for the Hausdorff dimension of $C_\delta$ is given 
  by the unique $0<s<1$ satisfying 
\begin{equation}\label{eq:exact}
    \left(\frac{\delta}{2}\right)^s+\left(\frac{1}{2}\right)^s=1,
    \qquad\text{or}\qquad 2^s=1+\delta^s.
\end{equation}
This follows from the easily verified fact that the Hausdorff measure of the 
  set $C_\delta$ in dimension $h(\delta)$ is equal to one.  
Note that Theorems~\ref{thm:gen:upper} and \ref{thm:gen:lower:1} both give 
  $h(\delta)$ as upper and lower bounds, respectively.  

From (\ref{eq:exact}) we have $$s=\log_2(1+\delta^s)\asymp\delta^s$$ so that 
\begin{align*}
  \log s &= s\log\delta + O(1), \\
  \frac1s\log\frac1s &\simeq \log\frac{1}{\delta}, \\
  \frac1s &\simeq \log\frac{1}{\delta}\log\frac{1}{\delta},
\end{align*}
  where $A\simeq B$ means the ratio tends to one as $\delta\to0$.  
It follows that 
\begin{equation}\label{ieq:h(delta)}
  h(\delta) \simeq \frac{\log\log(1/\delta)}{\log(1/\delta)}.  
\end{equation}

It is easy to check that the graph of $h(\delta)$ is \emph{nearly flat} as 
  $\delta\to1$ in the sense that all one-sided derivatives vanish there.  
Likewise, it is \emph{nearly vertical} as $\delta\to0$ in the sense that 
  the graph of the inverse function is nearly flat at the origin.  
Thus, we see that the both estimates $h_d(\delta)$ and $h_g(\delta)$ leave 
  plenty of room for improvement in the either limiting case $\delta\to0$ 
  or $\delta\to1$.  
One reason why $h_d(\delta)$ and $h_g(\delta)$ fail to provide sharp lower 
  bounds in the case of $C_\delta$ may be explained by the fact that the 
  lengths the intervals (or the gaps) in $E_k$ are not uniformly bounded: 
  $$\lim_{k\to\infty} \frac{d^+_k}{d^-_k} = \infty.$$  
It seems clear that better estimates should be obtainable if the parameters 
  $\Delta_k$, $d^\pm_k$, $m_k$, $\eps_k$ were replaced by parameters that 
  also depended on the particular interval in $E_k$; in other words, the 
  new parameters would be functions on a tree of intervals.  
Estimates on Hausdorff dimension would then be given in terms of ``local 
  conditions'' to be met at each node of this tree.  
An example of such condition is given by (\ref{ieq:lower:1}).  

Theorem~\ref{thm:gen:lower:1} is a special case of the next theorem, which 
  allows for a spacing condition with \emph{weights}.  

\begin{theorem}\label{thm:gen:lower:2}
Suppose (i) and (ii) of Theorem~\ref{thm:gen:lower:1} holds, and there exists 
  a function $\rho:J\to(0,1)$ such that 
\begin{enumerate}
\item[(iii')] for any $\alpha\in J$ and 
   for any distinct pair $\alpha',\alpha''\in\sigma(\alpha)$ 
\begin{equation}\label{ieq:spacing:2}
         \dist(B(\alpha'),B(\alpha''))>\rho(\alpha)\diam B(\alpha),  
\end{equation}
 \item[(iv')] there exists $s>0$ such that for every $\alpha\in J$, 
        $$\sum_{\alpha'\in\sigma(\alpha)} 
               [\rho(\alpha')\diam B(\alpha')]^s 
                  \ge [\rho(\alpha)\diam B(\alpha)]^s.$$  
\end{enumerate}
Then $\Hdim X\ge s$.  
\end{theorem}
\begin{proof}
Fix any $\alpha_0\in J$ and let $$E=E(\alpha_0)$$ be the set of all $x\in X$ such 
  that $\cap B(\alpha_j)=\{x\}$ for some $\sigma$-admissible sequence $(\alpha_j)$ 
  starting with $\alpha_0$.  
Let $J_0=\{\alpha_0\}$ and $J_k=\cup_{\alpha\in J_{k-1}}\sigma(\alpha)$ for $k>0$.  
Note that $$E=\cap_{k\ge0}E_k \quad\text{where}\quad E_k=\cup_{\alpha\in J_k}B(\alpha).$$  
Since each $J_k$ is finite, $E$ is compact.  Let $J'=\cup_{k\ge0}J_k$.  

\textbf{Claim.} For any finite subset $F\subset J'$ such that 
  $\cB_F=\{B(\alpha)\}_{\alpha\in F}$ covers $E$ we have 
\begin{equation}\label{ieq:finite}
  \sum_{\alpha\in F}[\rho(\alpha)\diam B(\alpha)]^s \ge [\rho(\alpha_0)\diam B(\alpha_0)]^s.  
\end{equation}
To prove the claim, it is enough to consider the case where $\cB_F$ has no redundant 
  elements, i.e. $B(\alpha)\cap E\neq\zip$ for all $\alpha\in F$, and 
  $B(\alpha)\not\subset B(\alpha')$ for any distinct pair $\alpha,\alpha'\in F$.  
It follows by (i) that the elements of $\cB_F$ form a disjoint collection.  

Proceed by induction on the smallest $k$ such that $F\subset J_0\cup\dots\cup J_k$.  
If $k=0$, then $F=\{\alpha_0\}$ and (\ref{ieq:finite}) holds with equality.  
For $k>0$, first note that for any $\alpha'\in F\cap J_k$ we have $\sigma(\alpha)\subset F$ 
  where $\alpha$ is the unique element of $J_{k-1}$ such that $\alpha'\in\sigma(\alpha)$.  
Indeed, given $\alpha''\in\sigma(\alpha)$, $B(\alpha'')\cap E\neq\zip$ implies 
  that $B(\alpha'')$ intersects $B(\alpha''')$ for some $\alpha'''\in F$.  
We cannot have $B(\alpha'')\subset B(\alpha''')$ for otherwise $B(\alpha')$ would be 
  a redundant element in $\cB_F$; therefore, $\alpha'''\not\in J_i$ for any $i<k$.  
Since $\alpha'''\in F$, we have $\alpha'''\in J_k$ so that $\alpha''=\alpha'''\in F$.  

Let $F'=F\cap(J_0\cup\dots\cup J_{k-1})$ and let $\Tilde F$ be the subset of $J_{k-1}$ 
  such that $F\cap J_k$ is the disjoint union of $\sigma(\alpha)$ as $\alpha$ ranges 
  over the elements of $\Tilde F$.  
Then (iv) implies   
  $$\sum_{\alpha\in F}[\rho(\alpha)\diam B(\alpha)]^s \ge 
       \sum_{\alpha\in F'\cup\Tilde F}[\rho(\alpha)\diam B(\alpha)]^s$$ 
  and the claim follows by the induction hypothesis applied to $F'\cup\Tilde F$.  

Now suppose $\cU$ is a covering of $E$ by open balls of radius at most $\eps$.  
Since $E$ is compact, there is a finite subcover $\cU_0$ and without loss of 
  generality we may assume each element of $\cU_0$ contains some point of $E$.  
For each $U\in\cU_0$ let $(\alpha_k)$ be the sequence determined by a choice of 
  $x\in U\cap E$ and requiring $x\in B(\alpha_k), \alpha_k\in J_k$ for all $k$.  
Let $k_0$ be the largest index $k$ such that $U\cap E\subset B(\alpha_k)$.  
Then there are distinct elements $\alpha',\alpha''\in \sigma(\alpha_k)$ such that 
  $U$ intersects both $B(\alpha')$ and $B(\alpha'')$ so that (iii) implies 
  $$\diam U \ge \dist(B(\alpha'),B(\alpha''))\ge \rho(\alpha_k)\diam B(\alpha_k).$$  
Let $F$ be the collection of $\alpha_k$ associated to $U\in\cU_0$.  Then 
  $$\sum_{U\in\cU} (\diam U)^s \ge \sum_{\alpha\in F}[\rho(\alpha)\diam B(\alpha)]^s 
        \ge [\rho(\alpha_0)\diam B(\alpha_0)]^s.$$  
Since $\eps>0$ was arbitrary, it follows that $E$ has positive $s$-dimensional 
  Hausdorff measure, and therefore, $\Hdim X\ge\Hdim E\ge s$.  
\end{proof}

\subsection{Divergent partial quotients}
As an illustration of the use of Theorems~\ref{thm:gen:upper} and 
  \ref{thm:gen:lower:1} we give an application to number theory.  
The rest of this section is independent of the other parts of the 
  paper and may be skipped without loss of continuity.  

We say an real number has \emph{divergent partial quotients} if it is 
  irrational and the sequence of terms $a_k$ in its continued fraction 
  expansion tends to infinity as $k\to\infty$.  
Let $D_\infty$ be the set of real numbers with divergent partial quotients.  
Our goal is to determine the Hausdorff dimension of $D_\infty$.  

Let $D_N$ be the set of all irrational numbers whose sequence of partial 
  quotients satisfy $a_k>N$ for all sufficiently large $k$.  
Given $p/q\in\Q$ with $q\ge2$ let $p_-/q_-<p_+/q_+$ be the convergents that 
  precede $p/q$ in the two possible continued fraction expansions for $p/q$.  
They are determined by the conditions $$p_\pm q -pq_\pm = \pm1, \quad 0<q_\pm<q$$ 
  and we note that $q=q_++q_-$.  
Let $v=(p,q)$ and set 
  $$I_N(v) = \left[\frac{Np+p_-}{Nq+q_-},\frac{Np+p_+}{Nq+q_+}\right].$$  
This interval consists of all real numbers that have $p/q$ as a convergent 
  and such that the next partial quotient is at least $N$.  
(In particular, we note $I_1(v)=\overline{\Delta(v)}$.)  
We have 
\begin{align*}
  \big|I_N(v)\big| &= \left|\frac{Np+p_-}{Nq+q_-} - \frac{p}{q}\right| 
                    + \left|\frac{p}{q} - \frac{Np+p_+}{Nq+q_+}\right| \\
     &= \frac{1}{(Nq+q_-)q} + \frac{1}{(Nq+q_+)q} = \frac{2N+1}{(Nq+q_-)(Nq+q_+)}
\end{align*}
  so that $$\frac{2}{(N+1)q^2}\le\big|I_N(v)\big|\le\frac{2}{Nq^2}.$$  

Let $\cB_N$ be the collection of intervals $I_N(v)$ for $v\in Q$, $|v|\ge2$.  
Let $\sigma_N(v)$ be the set of all $v'\in Q$ of the form $av+v_\pm$ where 
  $v_\pm=(p_\pm,q_\pm)$ and $a>N$.  
Then $(\cB_N,Q,\sigma_N)$ is a self-similar covering of $D_N$: the map $\cE$ is 
  realised by sending $x\in D_N$ to a tail of the sequence of convergents of $x$.  
For any $v'\in\sigma_N(v)$ we have $$\frac{N}{(a+1)^2(N+1)} \le 
        \frac{\big|I_N(v')\big|}{\big|I_N(v)\big|} \le \frac{N+1}{a^2N}$$ 
   where $a$ is the greatest integer less than $\tfrac{|v'|}{|v|}$.  
Note that there are two elements of $\sigma(v)$ associated with each $a>N$.  

To estimate Hausdorff dimension we need to consider the expression 
  $$\sum_{v'\in\sigma(v)}\frac{\big|I_N(v')\big|^s}{\big|I_N(v)\big|^s}.$$ 
For any $0<s<1$ we have 
\begin{align*}
  \sum_{a>N}\frac{2(N+1)^s}{a^{2s}N^s}\le\frac{4}{(2s-1)N^{2s-1}}
\end{align*}
  which is $\le1$ provided $$\log N\ge\frac{1}{2s-1}\log\frac{4}{2s-1}.$$  
Let $s_+=s_+(N)$ be the unique $s$ such that the above holds with equality.  
Note that if $y=x\log x$ then in the limit as $x\to\infty$ we have 
  $\log y\simeq\log x$ so that $x\simeq\frac{y}{\log y}$.  
It follows that, in the limit as $N\to\infty$ we have 
  $$\frac{1}{2s_+-1}\simeq\frac{\log N}{\log\log N}$$ so that 
  applying Theorem~\ref{thm:gen:upper} we now get 
  $$\Hdim D_N \le \frac{1}{2} + \frac{c\log\log N}{\log N}$$  
  for some constant $c>0$.  

Let $\sigma'_N(v)$ be the subset of $\sigma_N(v)$ consisting of those 
  $v'$ for which $\lfloor\tfrac{|v'|}{|v|}\rfloor\le2N$.  
For distinct $v',v''\in\sigma'_N(v)$ with $v'=(p',q'),v''=(p'',q'')$ we have 
  $$\left|\frac{p'}{q'}-\frac{p''}{q''}\right|\ge\frac{1}{q'q''}\ge\frac{1}{(2N+1)^2q^2}$$ 
  whereas $$\big|I(v')\big|\le\frac{2}{N(q')^2}\le\frac{2}{N^3q^2}$$ from which 
  we see that the gap between $I(v')$ and $I(v'')$ is at least (assuming $N\ge72$) 
     $$\frac{1}{9N^2q^2} - \frac{4}{N^3q^2} \ge \frac{1}{18N^2q^2} \ge 
        \frac{1}{36N}\big|I_N(v)\big|.$$  
Thus, (\ref{ieq:spacing:1}) holds with $\rho=\frac{1}{36N}$.  
For any $0<s<1$ we have (using $N\ge2$) 
\begin{align*}
  \sum_{N<a\le2N}\frac{2N^s}{(a+1)^{2s}(N+1)^s}
     &\ge \frac{1}{2s-1}\left(\frac{1}{(N+2)^{2s-1}}-\frac{1}{(2N+2)^{2s-1}}\right) \\ 
     &\ge \frac{1}{3(2s-1)(N+2)^{2s-1}} \ge \frac{1}{6(2s-1)N^{2s-1}}
\end{align*}
  which is $\ge1$ provided $$\log N\le\frac{1}{2s-1}\log\frac{1/6}{2s-1}.$$  
Let $s_-=s_-(N)$ be the unique $s$ such that the above holds with equality.  
It follows that in the limit as $N\to\infty$ we have 
  $$\frac{1}{2s_--1}\simeq\frac{\log N}{\log\log N}$$ so that 
  applying Theorem~\ref{thm:gen:lower:1} we now get 
  $$\Hdim D_N \ge \frac{1}{2} + \frac{c\log\log N}{\log N}$$  
  for some constant $c>0$.  

This establishes the following.  
\begin{theorem}\label{thm:dpq:2}
There are $c_2>c_1>0$ such that for all $N\ge72$ 
  $$\frac{1}{2} + \frac{c_1\log\log N}{\log N} \le 
       \Hdim D_N \le \frac{1}{2} + \frac{c_2\log\log N}{\log N}.$$  
\end{theorem}

Since $D_\infty=\cap_ND_N$, it follows that $\Hdim D_\infty\le\tfrac12$.  
To obtain the opposite inequality, one can repeat the argument for 
  the lower bound on $\Hdim D_N$ with $N$ ``replaced'' by a sequence 
  $N_k$ that slowly increases to infinity.  
The main issue is that the spacing condition (\ref{ieq:spacing:1}) 
  is no longer satisfied, and this is precisely the point where 
  Theorem~\ref{thm:gen:lower:2} is needed to complete the argument.  
We shall omit the details, since we can instead appeal to a classical 
  result of Jarnik-Besicovitch.  
Let $A_\delta$ ($\delta>0$) be the set of irrationals whose sequence 
  of partial quotients satisfy $a_{k+1}>q_k^\delta$, where $q_k$ is 
  the height of the $k$th convergent.  
Then $A_\delta\subset D_\infty$ and the theorem of Jarnik-Besicovitch 
  asserts that $$\Hdim A_\delta\ge\frac{1}{2+\delta}.$$  
Thus, we have established 
\begin{theorem}\label{thm:dpq}
The Hausdorff dimension of $D_\infty$ is $\tfrac12$.  
\end{theorem}

To the best of the author's knowledge, the results in this section 
  have not appeared in the literature.

\section{Upper bound calculation}\label{S:upper}
In the rest of the paper, we assume $d=2$.  

\begin{definition}\label{def:eps(v)}
For each $v\in Q$, let $$\eps(v)^{3/2}:=\frac{|L(v)|}{|v|^{1/2}}$$ 
  and define $$Q_\eps:=\{v\in Q: \eps(v)<\eps \}.$$  
\end{definition}

\begin{definition}\label{def:Hat:L(v)}
For each $v\in Q$, let $$\cL^*(v):=\cL(v)\setminus\{L(v)\}.$$  
Fix, once and for all, an element $\Hat L\in \cL^*(v)$ such that 
  $|\Hat L|$ is minimal, and denote this element by $$\Hat L(v).$$  
\end{definition}

An important consequence of the assumption $d=2$ is the following.  
\begin{lemma}\label{lem:Hat:L}
For any $v\in Q_\eps$, 
\begin{equation}\label{ieq:LLv}
      \frac{|v|}{|L(v)|} \le |\Hat L(v)| \le (1+\eps^3)\frac{|v|}{|L(v)|}  
\end{equation}
  and 
\begin{equation}\label{ieq:Hat:L:eps} 
  |\Hat L(v)| > \eps^{-3/2}|v|^{1/2} \quad\text{and}\quad 
  |\Hat L(v)| > \eps^{-3}|L(v)|.
\end{equation}
\end{lemma}
\begin{proof}
Let $L=L(v)$ and $\Hat L=\Hat L(v)$.  
We may think of them as vectors in the plane of lengths $|L|$ and $|\Hat L|$, 
  respectively, such that the area of the lattice they span is $|v|$.  
Without loss of generality we assume $L$ is horizontal.  
The vertical component of $\Hat L$ is then $\frac{|v|}{|L|}$ so that 
  $$\frac{|v|}{|L|} \le |\Hat L| \le \frac{|v|}{|L|} + |L|$$ 
  giving (\ref{ieq:LLv}).  
If $|L|<\eps^{3/2}|v|^{1/2}$ then the vertical component of $\Hat L$ 
  is greater than $\eps^{-3/2}|v|^{1/2}$, giving the first inequality 
  in (\ref{ieq:Hat:L:eps}).  
From this and $v\in Q_\eps$, the second inequality in (\ref{ieq:Hat:L:eps}) 
  follows.  
\end{proof}

By Theorem~\ref{thm:char}, for any $\bx\in\Sing^*(d)$ and any $\eps>0$ the 
  elements $v\in\Sigma(\bx)$ belong to $Q_\eps$ if $|v|$ is large enough.  
Let $\cB_\eps=\{\Delta(v)\}_{v\in Q_\eps}$ and define 
  $$\sigma_\eps\subset Q_\eps\times Q_\eps$$ 
  to consist of all pairs $(v,v')$ such that $|v|<|v'|$ and $v$ and $v'$ 
  realise a consecutive pair of local minima of the function $W_\bx$ for 
  some $\bx\in\Sing^*(d)$.  
For each $\bx\in\R^d$, we fix, once and for all, a sequence $(v_j)$ in 
  $\Sigma(\bx)$ such that each local minimum of $W_\bx$ is realised by 
  exactly one $v_j$, and, by an abuse of notation, we shall denote this 
  sequence by the same symbol $$\Sigma(\bx).$$  
Then it is easy to see that $(\cB_\eps,Q_\eps,\sigma_\eps)$ is a self-similar 
  covering of $\Sing^*(d)$ for any $\eps>0$: the map $\cE$ can be realised 
  by sending $\bx$ to a tail of $\Sigma(\bx)$.  
However, Theorem~\ref{thm:gen:upper} does not yield any upper bound because 
  it happens that $s(\cB_\eps,Q_\eps,\sigma_\eps)=\infty$ for all $\eps>0$, 
  the main reason being that there is no way to bound the ratios 
  $$\frac{\diam\Delta(v')}{\diam\Delta(v)}$$ away from one.

\subsection{First acceleration}

\begin{definition}\label{def:Hat:Sigma}
Suppose that $\Sigma(\bx)=(v_j)$.  We define $$\Hat\Sigma(\bx)$$ to be 
  the subsequence of $\Sigma(\bx)$ consisting of those $v_{j+1}$ 
  such that $$v_{j+1}\notin \Z v_j + \Z v_{j-1}.$$  
\end{definition}

\begin{lemma}\label{lem:Hat:inf}
The sequence $\Hat\Sigma(\bx)$ has infinite length iff $\bx$ does not lie 
  on a rational line in $\R^d$.  
\end{lemma}
\begin{proof}
The sequence $\Hat\Sigma(\bx)$ is finite iff there is a two-dimensional 
  sublattice $L\subset\Z^{d+1}$ such that $v_j\in L$ for all large enough $j$.  
If $\bx$ lies on the line $\ell\subset\R^d$ containing $\Dot v$ for all 
  $v\in L$ then the shortest vector in $g_th_\bx\Z^{d+1}$ is realised by 
  some vector in $g_th_\bx L$ for all large enough $t$.  
Conversely, $v_j\in L$ for all large enough $j$ implies 
  $\bx=\lim\Dot v_j\in\ell$.  
\end{proof}

\begin{lemma}\label{lem:Hat:char}
Let $\Hat\Sigma(\bx)=(u_k)$ where $\bx$ does not lie on a rational line.  
Let $\delta>0$ and $\eps>0$ be related by $\delta=\eps^{1+1/d}$.  
Then $\bx\in\DI_{\delta/2}(d)$ implies $\eps(u_k)<\eps$ for all sufficiently 
  large $k$, which in turn implies $\bx\in\DI_{2\delta}(d)$.  
In particular, $\bx\in\Sing(d)$ iff $\eps(u_k)\to0$ as $k\to\infty$.  
\end{lemma}
\begin{proof}
Let $\Sigma(\bx)=(v_j)$ and for each $j$, let $L_j=\Z v_j+\Z v_{j-1}$ and 
  $$\eps_j=\frac{|L_j|}{|v_j|^{1-1/d}}.$$  
Given $u_k$ there are indices $i<j$ such that $u_k=v_i,\dots,v_j=u_{k+1}$.  
Suppose that $j>i+1$.  Then $v_{i+1}\in L_i$ so that $L_{i+1}\subset L_i$ 
  and we have equality by Lemma~\ref{lem:consec}.  
It follows by induction that $L_{j-1}=\dots=L_i$, from which is follows 
  that $\eps_i>\dots>\eps_{j-1}$.   
Therefore, $\eps(u_k)<\eps$ for all sufficiently large $k$ iff 
  $\eps(v_j)<\eps$ for all sufficiently large $j$.  
The lemma now follows from Theorem~\ref{thm:char}.  
\end{proof}

\begin{definition}\label{def:V(u)}
For any $u\in Q$, let $$\cV(u)$$ be the set of vectors in $Q$ of the form 
  $au+b\Tilde v$ where $a,b$ are relatively prime integers such that $|b|\le a$ 
  and $\Tilde v\in Q$ is a vector that satisfies $L(u)=\Z u+\Z\Tilde v$, 
  $|\Tilde v|<|u|$ and $L(\Tilde v)\neq L(u)$.  
The set $$\cV_\eps(u)$$ is defined similarly, except that we additionally 
  require $\Tilde v\in Q_\eps$ and $|\Tilde v|<\eps|u|$.  
\end{definition}

\begin{lemma}\label{lem:Hat:consec}
Let $u,u'\in\Hat\Sigma(\bx)$ be consecutive elements with $|u|<|u'|$.  
Let $\Tilde v,v\in\Sigma(\bx)$ be the elements that immediately precede 
  $u$ and $u'$, respectively.  
Then $v\in\cV(u)$, provided $L(u)=\Z u+\Z\Tilde v\neq L(\Tilde v)$.  
\end{lemma}
\begin{proof}
Suppose $\Sigma(\bx)=(v_j)$ so that $u=v_i$, $v=v_j$, $\Tilde v=v_{i-1}$ 
  and $u'=v_{j+1}$ for some indices $i\le j$.  
We shall argue by induction on $j\ge i$ to show that 
\begin{enumerate}
  \item[(i)] $v_j = a u + b\Tilde v$ for some integers $-a<b\le a$, and 
  \item[(ii)] $v_j - v_{j-1} = a'u + b'\Tilde v$ for some integers $-a'\le b'\le a'$.  
\end{enumerate}
This is clear if $j=i$.  
For $j>i$, since $v_j\in\Z v_{j-1}+\Z v_{j-2}$, Lemma~\ref{lem:consec} implies 
  that $(v_j,v_{j-1})$ is an integral basis for $\Z v_{j-1}+\Z v_{j-2}$.  
Since $|v_j|>|v_{j-1}|>|v_{j-2}|$, there is a positive integer $c$ such that 
  $$v_j=cv_{j-1}+v_{j-2}, \quad\text{or}\quad  v_j=cv_{j-1}+(v_{j-1}-v_{j-2}).$$  
In either case $v_j=mu+n\Tilde v$ where $(m,n)=(ca+a',cb+b')$ satisfies 
  $-m<n\le m$ by the induction hypothesis.  
Similarly, $v_j-v_{j-1}=m'u+n'\Tilde v$ where $(m',n')=(ca-a',cb-b')$ satisfies 
  $-m\le n\le m$, again, by the induction hypothesis.  
The lemma now follows from (i).  
\end{proof}

\begin{definition}\label{def:Hat:sigma}
For any $\eps>0$, define $$\Hat\sigma_\eps\subset Q_\eps\times Q_\eps$$ 
  to be the set consisting of pairs $(u,u')$ for which there exists 
  $v\in\cV_\eps(u)$ such that $L(u')=\Z u'+\Z v\in\cL^*(v)$.  
\end{definition}

\begin{corollary}\label{cor:cover}
Let $\delta=\eps^{3/2}$ where $0<\eps<1$.  Then $(\cB_\eps,Q_\eps,\Hat\sigma_\eps)$ 
  is a self-similar covering of $\DI_{\delta/2}(2)$.  
\end{corollary}
\begin{proof}
Let $\Sigma(\bx)=(v_j)$ for a given $\bx\in\DI_{\delta/2}(2)$.  
Lemma~\ref{lem:Hat:char} implies for all large enough $j$ we have 
  $$\|\Z v_j+\Z v_{j-1}\|_{\cL(v_j)} = 
              \frac{|v_{j-1}\wedge v_j|}{|v_j|^{1/2}}<\eps^{3/2}<1$$ 
  so that $L(v_j)=\Z v_j+\Z v_{j-1}$ and $\eps(v_j)<\eps$.  
Lemma~\ref{lem:Hat:consec} now implies $\Hat\Sigma(\bx)$ is 
  eventually $\Hat\sigma_\eps$-admissible.  
\end{proof}

Our next task is to estimate $s(\cB_\eps,Q_\eps,\Hat\sigma_\eps)$.  
For this, we need to enumerate the elements of $\Hat\sigma_\eps(u)$ 
  for any given $u\in Q_\eps$.  
\begin{definition}\label{def:U(v,L')}
For any $v\in Q$ and $L'\in\cL^*(v)$, let $$\cU_\eps(v,L')$$ 
  be the set of $u'\in Q_\eps$ such that $L'=L(u')=\Z u'+\Z v$.  
\end{definition}
Note that, by definition, for any $u'\in\Hat\sigma_\eps(u)$ there exists 
  a $v\in\cV_\eps(u)$ and an $L'\in\cL^*(v)$ such that $u'\in\cU_\eps(v,L')$.  
Hence, for any $f:Q_\eps\to\R_+$ 
  $$\sum_{u'\in\Hat\sigma_\eps(u)}f(u')\le \sum_{v\in\cV_\eps(u)} 
       \sum_{L'\in\cL^*(v)} \sum_{u'\in\cU_\eps(v,L')} f(u').$$  

\begin{notation}
We write $A\ls B$ to mean $A\le CB$ for some universal constant $C$.  
Note that $A\asymp B$ is equivalent to $A\ls B$ and $B\ls A$.  
We write $A\gs B$ to mean the same thing as $B\ls A$.  
\end{notation}

\begin{proposition}\label{prop:key}
There is a constant $C$ such that for any $u\in Q_\eps$, and 
  for any $s>4/3$ and any $r<6s-3$ 
\begin{equation}\label{ieq:key}
  \sum_{u'\in\Hat\sigma_\eps(u)} \left(\frac{\eps(u)}{\eps(u')}\right)^r
     \left(\frac{\diam\Delta(u')}{\diam\Delta(u)}\right)^s 
      \le \frac{C(6s-3-r)^{-1}\eps^{6s-3-r}}{(3s-4)^2(\eps(u))^{6-3s-r}}.  
\end{equation}
\end{proposition}
\begin{proof}
Given $v\in\cV_\eps(u)$ there are $a,b\in\Z$ with $|b|<a$ and $\Tilde v\in Q_\eps$ 
  such that $v=au+b\Tilde v$, $|\Tilde v|<\eps|u|$, $L(u)=\Z u+\Z\Tilde v$ 
  and $L(u)\neq L(\Tilde v)$.  
Note that $\Tilde v\in Q_\eps$ implies $|\Hat L(\Tilde v)|>\eps^{-3/2}|\Tilde v|^{1/2}$ 
  and since $L(u)\neq L(\Tilde v)$, we have $|L(u)|\ge|\Hat L(\Tilde v)|$, 
  so that $u\in Q_\eps$ now implies $$|u| > \eps^{-3}|L(u)|^2 > \eps^{-6}|v|$$  
  so that $$|v|\asymp a|u|.$$  
Since $|\Tilde v|<|u|$ and $(u,\Tilde v)$ is an integral basis for $L(u)$, 
  there are at most two possibilities for $\Tilde v$, so that given $u$ and 
  the positive integer $a$, there are at most $O(a)$ possibilities for $v$.  
It follows that for any $q>2$ 
\begin{equation}
\label{ieq:sum:v}
 \sum_{v\in\cV_\eps(u)} \frac{|u|^q}{|v|^q} \ls \sum_{a}\frac{1}{a^{q-1}} \asymp \frac{1}{q-2}.
\end{equation}

Let $\cL_+(v)$ be the set of elements in $\cL(v)$ considered with 
  orientations, and think of it as a subset of $\Wedge^2\Z^3$.  
Note that addition is defined for those pairs $L,L'\in\cL_+(v)$ whose 
  $\Z$-span contains all of $\cL_+(v)$.  
Let $L$ and $\Hat L$ be the elements in $\cL_+(v)$ corresponding to a 
  fixed choice of orientation for $L(v)$ and $\Hat L(v)$, respectively.  
Each $L'\in\cL^*(v)$ can be oriented so that, as an element in $\cL_+(v)$ 
  we have $L'=\Tilde a\Hat L+\Tilde b L$ for some (relatively prime) 
  integers $\Tilde a,\Tilde b$ with $\Tilde a>0$.  
Let $$\Hat L_m=\Hat L + m L,\quad m\in\Z.$$  
There is a unique integer $m$ such that $L'=L_m$ or $L'$ is a 
  postive linear combination of $\Hat L_m$ and $\Hat L_{m+1}$.  
In any case, for each $L'\in\cL^*(v)$ there is an integer $m$ (and an 
  orientation for $L'$) such that $$L'=a'\Hat L_m+b'L \quad a'>b'\ge0.$$  
Note that $v\in Q_\eps$ implies 
  $$|\Hat L_m|\ge |\Hat L| > \eps^{-3/2}|v|^{1/2} > \eps^{-3}|L|$$ 
  so that $$|L'|\asymp a'|\hat L_m|.$$  
Let $N=\left\lfloor\frac{|\Hat L|}{|L|}\right\rfloor$ so that 
  $|\Hat L_m|\asymp|L|(N+|m|)$ and 
\begin{align*}
  \sum_{m\in\Z}\frac{|\Hat L|^p}{|\Hat L_m|^p}
        &\asymp \frac{|\Hat L|^p}{|L|^p}\sum_{m\ge1}\frac{1}{(N+m)^p} \\
        &\asymp \frac{|\Hat L|^p}{|L|^p}\left(\sum_{m=1}^N\frac{1}{N^p} 
                     + \sum_{m>N}\frac{1}{m^p}\right)\\ 
        &\asymp \frac{|\Hat L|^p}{|L|^p}\left(\frac{1}{N^{p-1}} 
                     + \frac{1}{(p-1)N^{p-1}}\right) 
         \asymp \frac{p}{p-1}\frac{|\Hat L|}{|L|}.  
\end{align*}
Since there are at most $O(a')$ possibilities for $L'$ given $v$, $m$ and $a'$, 
  and since $|L||\Hat L|\asymp|v|$, it follows that for any $q>2$ 
\begin{equation}
\label{ieq:sum:L'}
  \sum_{L'\in\cL^*(v)} \frac{1}{|L'|^q} \ls 
      \sum_{m\in\Z}\frac{1}{|\Hat L_m|^q} \sum_{a'}\frac{1}{(a')^{q-1}} 
      \asymp \frac{1}{q-2}\frac{|L|^{q-2}}{|v|^{q-1}}.  
\end{equation}

Associate to each $u'\in\cU_\eps(v,L')$ the positive integer 
  $$c = \left\lceil\frac{|u'|}{|v|}\right\rceil > \frac{\eps^{-3}|L'|^2}{|v|} 
     \ge \eps^{-3}\frac{|\Hat L|^2}{|v|} \ge \eps^{-3}\frac{|v|}{|L|^2} > \eps^{-6}.$$ 
Since $|u'|\asymp c|v|$ and there are $2$ possibilities for $u'$, given $v$, $L'$ and $c$, 
  and since $c>\frac{\eps^{-3}|L'|^2}{|v|}$, it follows that for any $p>1$ 
\begin{equation}\label{ieq:sum:u'}
  \sum_{u'\in\cU_\eps(v,L')} \frac{1}{|u'|^p} 
    \asymp \frac{1}{|v|^p} \sum_c \frac{1}{c^p} 
    \asymp \frac{\eps^{3p-3}}{(p-1)|v||L'|^{2p-2}}.  
\end{equation}

Using (\ref{ieq:sum:u'}), (\ref{ieq:sum:L'}) and (\ref{ieq:sum:v}) 
  with $p=2s-\frac{r}{3}$ and $q=3s-2$, we obtain 
\begin{align*}
\sum_{v\in\cV_\eps(u)} \sum_{L'\in\cL^*(v)} \sum_{u'\in\cU_\eps(v,L')} 
   &\left(\frac{|L'|}{|L|}\right)^{s-\frac{2r}{3}} 
                     \left(\frac{|u|}{|u'|}\right)^{2s-\frac{r}{3}} \\ 
   \ls \sum_{v\in\cV_\eps(u)} \sum_{L'\in\cL^*(v)} 
       &\frac{(6s-3-r)^{-1}\eps^{6s-3-r}|u|^{2s-\frac{r}{3}}}
                           {|L|^{s-\frac{2r}{3}}|v||L'|^{3s-2}} \\
   \ls \sum_{v\in\cV_\eps(u)} 
       &\frac{(6s-3-r)^{-1}\eps^{6s-3-r}|u|^{2s-\frac{r}{3}}|L|^{2s-4+\frac{2r}{3}}}
                             {(3s-4)|v|^{3s-2}} \\
   \ls \quad &\frac{(6s-3-r)^{-1}\eps^{6s-3-r}|L|^{2s-4+\frac{2r}{3}}}
             {(3s-4)^2|u|^{s-2+\frac{r}{3}}} 
\end{align*}
  and this completes the proof of the proposition.  
\end{proof}

Proposition~\ref{prop:key} implies there exists $C>0$ such that for any $s>\frac{4}{3}$ 
\begin{equation}\label{ieq:first:accel}
  \sum_{u'\in\Hat\sigma_\eps(u)} \left(\frac{\diam\Delta(u')}{\diam\Delta(u)}\right)^s 
      \le \frac{C\eps^{6s-3}}{(3s-4)^2(\eps(u))^{6-3s}}.  
\end{equation}
However, this still does not imply $s(\cB_\eps,Q_\eps,\Hat\sigma_\eps)<\infty$ for any $\eps>0$.  

\subsection{Second acceleration}
Let $\Hat\sigma'_\eps\subset Q\times Q$ be the set of all pairs $(u,u')$ 
  satisfying $u\in Q_\eps$ and $u' \in \bigcup_{j\ge1}\sigma''_j(u)$ where 
\begin{align*}
 \sigma''_1(u) &:= \Hat\sigma_{\eps(u)}(u), & 
 \sigma''_j(u) &:= \bigcup\{\Hat\sigma_{\eps(u)}(u'):u'\in\sigma'_{j-1}(u)\}, \\
 \sigma'_1(u)  &:= \Hat\sigma_\eps(u), & 
 \sigma'_j(u)  &:= \bigcup\{\Hat\sigma_\eps(u'):u'\in\sigma'_{j-1}(u)\}.  
\end{align*}

\begin{proposition}\label{prop:hat:sigma'}
$s(\cB_\eps,Q_\eps,\Hat\sigma'_\eps)=\frac{4}{3}+O(\eps^{3/2})$.  
\end{proposition}
\begin{proof}
To simplify notation, we denote the diameter of a set by $|\cdot|$.  
Given $s>\frac{4}{3}$, we apply Proposition~\ref{prop:key} to ensure that 
  if $\eps>0$ is sufficiently small then, by (\ref{ieq:first:accel}), 
  for any $u\in Q_\eps$ 
\begin{align*}
  \sum_{u'\in\Hat\sigma_\eps(u)} &\frac{|\Delta(u')|^s}{|\Delta(u)|^s} \le 
    \frac{1}{2}\left(\frac{\eps}{\eps(u)}\right)^{6-3s} \quad \text{and} \\ 
  \sum_{u'\in\Hat\sigma_\eps(u)} &\left(\frac{\eps(u)}{\eps(u')}\right)^{6-3s}
    \frac{|\Delta(u')|^s}{|\Delta(u)|^s} \le \frac{C\eps^{9s-9}}{(3s-4)^2} 
    \le \frac{1}{2}.  
\end{align*}
We may choose $\eps$ so that $\eps^{9s-9}\asymp(3s-4)^2$; hence, 
  $s=\frac{4}{3}+O(\eps^{3/2}).$  

For each $u'\in\sigma''_j(u)$ there are $u_1,\dots,u_{j-1}\in Q_\eps$ such that 
  $$u_1\in\Hat\sigma_\eps(u),\;\; \dots,\;\; u_{j-1}\in\Hat\sigma_\eps(u_{j-2}), 
    \quad\text{and}\quad u'\in\Hat\sigma_{\eps(u)}(u_{j-1}).$$  
\begin{align*}
  \sum_{u'\in\sigma''_j(u)}\frac{|\Delta(u')|^s}{|\Delta(u)|^s} 
  &\le \sum_{u_1}\frac{|\Delta(u_1)|^s}{|\Delta(u)|^s} \dots 
       \sum_{u_{j-1}}\frac{|\Delta(u_{j-1})|^s}{|\Delta(u_{j-2})|^s} 
  \sum_{u'}\frac{|\Delta(u')|^s}{|\Delta(u_{j-1})|^s} \\ 
  &\le \frac{1}{2}\sum_{u_1}\frac{|\Delta(u_1)|^s}{|\Delta(u)|^s} \dots 
       \sum_{u_{j-1}}\frac{|\Delta(u_{j-1})|^s}{|\Delta(u_{j-2})|^s} 
       \left(\frac{\eps(u)}{\eps(u_{j-1})}\right)^{6-3s} \\ 
  &\le \frac{1}{2^2}\sum_{u_1}\frac{|\Delta(u_1)|^s}{|\Delta(u)|^s} \dots 
       \sum_{u_{j-2}}\frac{|\Delta(u_{j-2})|^s}{|\Delta(u_{j-3})|^s} 
       \left(\frac{\eps(u)}{\eps(u_{j-2})}\right)^{6-3s} \\ 
  &\le \dots \le \frac{1}{2^j}
\end{align*}
  so that   
  $$\sum_{u'\in\Hat\sigma'_\eps(u)}\frac{|\Delta(u')|^s}{|\Delta(u)|^s} \le 
    \sum_{j\ge1}\sum_{u'\in\sigma''_j(u)}\frac{|\Delta(u')|^s}{|\Delta(u)|^s} 
    \le \sum_{j\ge1}\frac{1}{2^j} \le 1$$ 
  and therefore $s(\cB_\eps,Q_\eps,\Hat\sigma_\eps)\le s$.  
\end{proof}

Now we verify that $(\cB_\eps,Q_\eps,\Hat\sigma'_\eps)$ is a self-similar 
  covering of $\Sing^*(2)$.  
Let $\cE(\bx)$ be a subsequence $(w_i)$ of $\Hat\Sigma(\bx)$ such that 
  $\eps(w_i)$ is strictly decreasing to zero as $i\to\infty$, where the 
  initial element $w_0$ is chosen so that for all $v$ that occurs after 
  $w_0$ in the sequence $\Sigma(\bx)$ we have $\eps(v)<\eps$.  
The sequence $\cE(\bx)$ is $\Hat\sigma'_\eps$-admissible by construction.  
It follows that $(\cB_\eps,Q_\eps,\Hat\sigma'_\eps)$ is a self-similar 
  covering of $\Sing^*(2)$, and by Theorem~\ref{thm:gen:upper} we can 
  now conclude that $$\Hdim \Sing^*(2)\le\frac{4}{3}.$$  

We now describe how the preceding argument can be modified to give 
  an upper bound estimate on $\Hdim \DI_\delta(2)$.  
Let $\DI_\delta^*(2)$ denote the set $\DI_\delta(2)$ with all rational 
  lines removed.  
Extend the map $\cE$ to all of $\DI_\delta^*(2)$ by choosing the 
  subsequence $(w_i)$ of $\Hat\Sigma(\bx)$ so that $\eps(w_i)$ is a 
  monotone sequence, and such that $\eps(w_0)<2\lim_i\eps(w_i)$.  
Modify the definition of $\Hat\sigma'_\eps$ by replacing the subscript $\eps(u)$ 
  in the formula for $\sigma''_j$, $j\ge1$ with the expression $2\eps(u)$.  
Then $(\cB_\eps,Q_\eps,\Hat\sigma_\eps)$ is a self-similar covering of 
  $\DI_\delta^*(2)$ provided $\delta=\frac{\eps^{3/2}}{2}$.  
Theorem~\ref{thm:gen:upper} and Proposition~\ref{prop:hat:sigma'} now 
  imply $$\Hdim \DI_\delta(2)=\frac{4}{3}+O(\delta).$$

\section{Lower bound calculation}\label{S:lower}

\begin{lemma}\label{lem:nested}
Let $0<\eps<\tfrac12$.  
Suppose $u\in Q$, $L'\in\cL^*(u)$, and $u'\in L'$ is such that 
  $L'=\Z u'+\Z u$ and $|u'|>\eps^{-3}|L'|^2$.  
Then $L'=L(u')$, $u'\in Q_\eps$, 
  and $$\overline{\Delta(u')}\subset\Delta(u).$$  
Moreover, if $u\in Q_\eps$ then $|u'|>\eps^{-6}|u|$.  
\end{lemma}
\begin{proof}
Since the norm of $L'\in\cL(u')$ is 
  $$\|L'\|_{\cL(u')} = \frac{|L'|}{|u'|^{1/2}} < \eps^{3/2} < 1$$ 
  we have $L'=L(u')$ and $\eps(u')<\eps$ so that $u'\in Q_\eps$.  
Let $L=L(u)$ and note that since $L'\neq L$, we have 
  $$|L'|\ge|\Hat L(u)|\ge\frac{|u|}{|L|}.$$  
Then 
  $$\dist(\Dot u,\Dot u') = \frac{|u\wedge u'|}{|u||u'|} 
     < \frac{\eps^3}{|u||L'|} \le \frac{\eps^3|L|}{|u|^2}.$$  
The fact that the Euclidean length of the shortest nonzero vector 
  in any two-dimensional unimodular lattice is universally bounded 
  above by $\sqrt{2}$ implies that 
\begin{equation}\label{ieq:eps<2}
  \eps(u)^3<2 \quad\text{for any}\quad u\in Q.  
\end{equation}
Therefore, 
  $$\frac{|L'||u|^2}{|L||u'|^2} \le \frac{\eps^6|u|^2}{|L||L'|^3} 
      \le \frac{\eps^6|L|^2}{|u|} < 2\eps^6$$ 
  so that 
  $$\frac{\eps|L|}{|u|^2}+\frac{2|L'|}{|u'|^2} 
      < (\eps^3+4\eps^6)\frac{|L|}{|u|^2} < \frac{|L|}{2|u|^2}.$$  
Theorem~\ref{thm:diam} now implies $\overline{\Delta(u')}\subset\Delta(u)$.  
If $u\in Q_\eps$ then 
  $$|u'|>\eps^{-3}|L'|^2>\frac{\eps^{-3}|u|^2}{|L|^2}>\eps^{-6}|u|.$$  
\end{proof}

\begin{definition}\label{def:N(u)}
For each $u\in Q$, let $$\cN_\eps(u)$$ be the set of $u'\in Q$ 
  such that $\Z u'+\Z u\in\cL^*(u)$ and $|u'|>\eps^{-3}|u\wedge u'|^2$.  
\end{definition}

Note that $\cN_\eps(u)\subset Q_\eps$, by definition.  

\begin{theorem}\label{thm:suff}
Let $0<\eps<\tfrac13$.  Suppose $(u_k)$ is a sequence in $Q$ satisfying 
  $u_{k+1}\in\cN_\eps(u_k)$ for all $k\ge0$.  Then 
\begin{enumerate}
\item[(a)] The limit $\bx:=\lim_k\Dot u_k$ exists and $u_k\in\Sigma(\bx)$ for all $k$.  
\item[(b)] $\bx\in\DI_\delta(2)$, where $\delta=2\eps^{3/2}$.  
\item[(c)] If $\eps(u_k)\to0$ as $k\to\infty$ then $\bx\in\Sing(2)$.  
\item[(d)] For all sufficiently large $t$ 
\begin{equation}\label{ieq:suff}
  W(t) - \log(1-\eps^6) \le W_\bx(t) \le W(t) 
\end{equation}
  where $$W(t)=\log\ell(g_th_\bx\{u_k\})=\log\min_{k\ge0}\|g_th_\bx u_k\|'.$$  
\end{enumerate}
\end{theorem}
\begin{proof}
Apply Lemma~\ref{lem:nested} with $L'=\Z u_{k+1}+\Z u_k$ to conclude that 
  $\cap_k\Delta(u_k)$ is nonempty, and $u_k\in Q_\eps$ for all $k\ge1$.  
Moreover, $|u_k|\to\infty$ so that $\diam\Delta(u_k)\to0$ and (a) follows.  
Lemma~\ref{lem:nested} also implies $L(u_{k+1})=\Z u_{k+1}+\Z u_k$, so that, 
  by Theorem~\ref{thm:best:1}, we have $$\eps_\bx(u_k,u_{k+1})^{3/2} \le 
  2\frac{|u_k\wedge u_{k+1}|}{|u_{k+1}|^{1/2}} = 2\eps(u_{k+1})^{3/2} < \delta.$$  
Hence, Lemma~\ref{lem:common} implies the local maxima of the piecewise 
  linear function $W$ are all bounded by $\log\delta$.  
Since $W_\bx\le W$ (by monotonicity of $\ell$), it follows that the local 
  maxima of $W_\bx$ are bounded by $\log\delta$, eventually.  
By Theorem~\ref{thm:char}, this means $\bx\in\DI_\delta(2)$, giving (b).  
If $\eps(u_k)\to0$ then $W_\bx(t)\to-\infty$, so that $\bx\in\Sing(2)$.  
  This proves (c).  

Since $W_\bx\le W$, the second inequality in (\ref{ieq:suff}) actually 
  holds for all $t$.  
For the first inequality, we consider a local maximum time $t$ for $W$.  
Then for some index $k$, if we set $u=g_th_\bx u_k$ and $u'=g_th_\bx u_{k+1}$ 
  then $\|u\|'=\|u'\|'$.  
The corresponding local maximum value is $\log\eps'$ where $\eps'$ 
  is the common $\|\cdot\|'$-length of $u$ and $u'$.  
To prove the first inequality in (\ref{ieq:suff}) it suffices to show 
  that for any $w\in g_th_\bx\Z^3$ we have 
\begin{equation}\label{ieq:short}
  \|w\|'\ge(1-\eps^6)\eps'.  
\end{equation}
Lemma~\ref{lem:nested} implies (for $k\ge1$) $|u'|>\eps^{-6}|u|$.  
Hence, $\|u'\pm u\|'\ge|u'\pm u|\ge(1-\eps^6)\eps'$.  Note that 
  $\|au'+bu\|'\ge\eps'$ for any pair of integers with $|a|\neq|b|$.  
This establishes (\ref{ieq:short}) for $w\in\Z u+\Z u'$.  
Let $\|\cdot\|_e$ denote the Euclidean norm on $\R^3$.  
For any $v\in\R^3$ we have $$\|v\|'\le\|v\|_e\le\sqrt{2}\|v\|'.$$  
The Euclidean area of a fundamental parallelogram for $\Z u+\Z u'$ 
  is $$\|u\wedge u'\|_e\le\|u\|_e\|u'\|_e\le2(\eps')^2$$ so that 
  for any $w\in g_th_\bx\Z^3\setminus(\Z u+\Z u')$ we have 
  $$\|w\|'\ge\frac{\|w\|_e}{\sqrt{2}}\ge\frac{1}{2\sqrt{2}(\eps')^2}$$ 
  which is $>\eps'$ since $$\eps' \le 3^{2/3}\eps(u_{k+1}) 
    < 3^{2/3}\eps < \frac{1}{\sqrt[3]{3}} < \frac{1}{\sqrt{2}}.$$  
Thus (\ref{ieq:short}) holds for all $w\in g_th_\bx\Z^3$.  
\end{proof}

We assume for each $u\in Q$ orientations for $L(u)$ and $\Hat L(u)$ have 
  been chosen so that we may think of them as elements of $\Wedge^2\Z^3$.  
\begin{definition}\label{def:psi}
Given $u\in Q$ and integers $a\ge b\ge 0$ such that $\gcd(a,b)=1$ we set 
  $$L'=a'\Hat L(u)+b'L(u).$$  
Additionally, given $0<\eps<1$ and an integer $c\ge1$ satisfying 
  $$M_\eps < c < 2M_\eps-1 \quad\text{where}\quad M_\eps=\frac{\eps^{-3}|L'|^2}{|u|}$$ 
  we define $\psi_\eps(u,a,b,c)$ to be the unique $u'\in Q$ such that 
  $$L'=u'\wedge u \quad\text{and}\quad \left\lfloor\frac{|u'|}{|u|}\right\rfloor=c.$$  
\end{definition}
Note that $\psi_\eps(u,a,b,c)\in\cN_\eps(u)$ because $L'=L(u')$.  
Note also that $c>M_\eps$ implies $\psi_\eps(u,a,b,c)\in Q_\eps$ 
  while $c<2M_\eps-1$ implies $\psi_\eps(u,a,b,c)\notin Q_{\eps/2}$.  
Therefore, we always have $$\psi_\eps(u,a,b,c)\in \cN_\eps(u)\cap Q'_\eps
  \quad\text{where}\quad Q'_\eps:=Q_\eps\setminus Q_{\eps/2}.$$  

\begin{lemma}\label{lem:spacing}
Let $0<\eps<2^{-7}$.  
If $u'=\psi_\eps(u,a,b,c)$ and $u''=\psi_\eps(u,a',b',c')$ are such that 
  $(a,b)\neq(a',b')$ or $|c-c'|\ge20$ then 
\begin{equation}\label{ieq:spacing}
  \dist(\Delta(u'),\Delta(u''))\ge\frac{\eps^9}{2^{11}N^3}\diam\Delta(u) 
\end{equation}
  where $N=\max(a,a')$.  
\end{lemma}
\begin{proof}
Let $L=L(u)$, $\Hat L=\Hat L(u)$ and $L'=u'\wedge u$.  
Note that by (\ref{ieq:LLv}) we have 
  $$ |L'| \le 2a|\Hat L| \le \frac{4N|u|}{|L|}.$$  
Theorem~\ref{thm:diam} implies $$\diam \Delta(u)\le \frac{4|L|}{|u|^2}$$ and 
  also $\Delta(u')\subset B(\Dot u',2r')$ where $$r'=\frac{|L'|}{|u'|^2} \le 
  \frac{|L'|}{c^2|u|^2} < \frac{|L'|}{M_\eps^2|u|^2} = \frac{\eps^6}{|L'|^3}.$$  
If $(a,b)=(a',b')$ then $$\dist(\Dot u',\Dot u'') 
      = \frac{|u'\wedge u''|}{|u'||u''|} > \frac{20|L'|}{(c+1)^2|u|^2} 
      > \frac{5|L'|}{M_\eps^2|u|^2} = \frac{5\eps^6}{|L'|^3}$$ 
  so that $$\dist(\Delta(u'),\Delta(u'')) \ge \frac{\eps^6}{|L'|^3} \ge 
    \frac{\eps^6|L|^3}{2^6N^3|u|^3} \ge \frac{\eps^9|L|}{2^9N^3|u|^2}$$ 
  giving (\ref{ieq:spacing}) in the case $(a,b)=(a',b')$.  
If $(a,b)\neq(a',b')$, let $L''=u''\wedge u$ and note that 
  $$\sin\angle\pi_u(L')\pi_u(L'')=\frac{|u|}{|L'||L''|} 
     \ge \frac{|L|^2}{16N^2|u|} \ge \frac{\eps^3}{2^7N^2}$$ and 
  $$\dist(\Dot u,\Dot u')=\frac{|u\wedge u'|}{|u||u'|} 
     \ge \frac{|L'|}{2M_\eps|u|^2} = \frac{\eps^3}{2|u||L'|} 
     \ge \frac{\eps^3|L|}{8N|u|^2}$$ so that 
  $$\dist(\Dot u',\Dot u'')\ge\frac{\eps^6|L|}{2^{10}N^3|u|^2}.$$  
Considering the component of $L'$ perpendicular to $L$, as in the 
  proof of Lemma~\ref{lem:Hat:L}, we get 
  $$|L'|\ge \frac{a|u|}{|L|} \ge \frac{N|u|}{|L|}$$ so that 
  $$2r'<\frac{2\eps^6}{|L'|^3} \le \frac{2^7\eps^6|L|^3}{N^3|u|^3} 
       < \frac{2^8\eps^6|L|}{|u|^2}$$ by (\ref{ieq:eps<2}).  
Since $\eps<2^{-7}$, it follows that $$\dist(\Delta(u'),\Delta(u'')) 
  \ge (\frac{1}{2^{10}}-2^9\eps^3)\frac{\eps^6|L|}{N^3|u|^2} 
  \ge \frac{\eps^6}{2^{13}N^3}\diam\Delta(u)$$ 
  which easily implies (\ref{ieq:spacing}).  
\end{proof}

The next proposition completes the proof of Theorem~\ref{thm:main:2}.  
\begin{proposition}\label{prop:DI(2):lb}
There is a constant $c>0$ such that for $0<\delta<2^{-10}$ 
  $$\Hdim \DI_\delta(2)\ge\frac43 + exp(-c\delta^{-4}).$$  
\end{proposition}
\begin{proof}
Fix a parameter $N$ to be determined later and set 
  $$\sigma_\eps(u)=\{\psi_\eps(u,a,b,c): a\le N,~20|c\}.$$  
Fix $u_0\in Q$, let $U_0=\{u_0\}$ and recursively define 
  $$U_{k+1}=\bigcup_{u\in U_k} \sigma_\eps(u)$$ where 
  $\eps$ is defined by $\delta=3\eps^{3/2}$.  
Note that $$E_k=\bigcup_{u\in U_k}\overline{\Delta(u)}$$ 
  is a disjoint union, by Lemma~\ref{lem:spacing}.  
We have $E_{k+1}\subset E_k$ by Lemma~\ref{lem:nested}, and 
  by Theorem~\ref{thm:suff}(a), there is a one-to-one 
  correspondence between the points of $E=\cap E_k$ and 
  the sequences $(u_k)$ starting with $u_0$ and satisfying 
  $u_{k+1}\in\sigma_\eps(u_k)$ for all $k$.  
Theorem~\ref{thm:suff}(b) implies $E\subset\DI_\delta(2)$.  
The hypotheses (i)-(iii) of Theorem~\ref{thm:gen:lower:1} 
  now hold with $$\rho=\frac{\eps^9}{2^{11}N^3}.$$  
Before checking (iv), we note that given $1\le a\le N$ 
  we have $\phi(a)$ choices for $b$ such that $a\ge b\ge0$ 
  and $\gcd(a,b)=1$, where $\phi$ is the Euler totient function.  
It is well known that $$\liminf_{n\to\infty}\frac{\phi(n)\log\log n}{n}>0.$$  
Now, for (iv) we compute (assuming $s>\tfrac43$) 
\begin{align*}
  \sum_{u'\in\sigma_\eps(u)}\frac{|L'|^s|u|^{2s}}{|L|^s|u'|^{2s}} 
     &\asymp \sum_{L'}\frac{|L'|^s}{|L|^s}\sum_c\frac{1}{c^{2s}} \\ 
     &\asymp \sum_{L'}\frac{|L'|^s}{|L|^s}\left(\frac{\eps^3|u|}{|L'|^2}\right)^{2s-1} \\
     &\asymp \frac{\eps^{6s-3}|u|^{2s-1}}{|L|^s|\Hat L|^{3s-2}}
             \sum_{a'}\frac{\phi(a')}{(a')^{3s-2}} \\
     &\gs \eps^{9s-6}\int_e^N \frac{dx}{x^{3s-3}\log x}.  
\end{align*}
Note that as $p\to1^+$ 
\begin{align*}
  \int_e^\infty\frac{dx}{x^p\log x} 
      &\asymp \sum_{k\ge1}\int_{e^k}^{e^{k+1}}\frac{dx}{x^pk} 
      = \sum_{k\ge1}\frac{e^{-k(p-1)}}{(p-1)k}(1-e^{-(p-1)}) \\ 
      &= \frac{1-e^{-(p-1)}}{p-1}\log\frac{1}{1-e^{-(p-1)}} 
      \asymp \log\frac{1}{p-1}.  
\end{align*}
Thus, we conclude that there is a constant $C>1$ such that 
  for any $s>\tfrac43$ satisfying 
  $$\eps^{9s-6}\left|\log\left(s-\frac{4}{3}\right)\right| > C$$ 
  the condition (iv) of Theorem~\ref{thm:gen:lower:1} holds 
  by choosing $N$ large enough (depending on $\eps$).  
Since $\delta=3\eps^{3/2}$, the proposition follows.  
\end{proof}

We now describe how to modify the preceding argument to 
  obtain the lower bound in Theorem~\ref{thm:main}.  
Fix a parameter $C>1$ to be determined later and choose 
  positive sequences $\eps_k\to0$ and $N_k\to\infty$ 
  such that for all $k$, $$\eps_k<2^{-7}, N_k\ge1, \quad
  \text{and}\quad \eps_k^6\log\log N_k>C.$$  
We shall also assume the sequences are slowly varying in the 
  sense that the ratio of consecutive terms are bounded above 
  and below by positive constants, say $2$ and $\tfrac12$.  
For example, 
  $$N_k=k+1, \qquad \eps_k=\frac{1}{2^7\log\log\log(k+C')}$$ 
  where $C'>1$ is chosen large enough depending only on $C$.  
The definition of the sets $U_k$ are modified by the formula 
  $$U_{k+1}=\bigcup_{u\in U_k} \sigma_{\eps_{k+1}}(u).$$  
With $E$ defined the same way as before, Theorem~\ref{thm:suff}(c) 
  now implies $E\subset\Sing(2)$.  
For each $u\in U_k$ set $$\rho(u)=\frac{\eps_{k+1}^9}{2^{11}N_k^3}$$  
  so that (i)-(iii) of Theorem~\ref{thm:gen:lower:2} hold.  
The main calculation in the proof of Proposition~\ref{prop:DI(2):lb} 
  with $s=\tfrac43$ now yields 
  $$\sum_{u'\in\sigma_{\eps_{k+1}}(u)}\frac{|L'|^s|u|^{2s}}{|L|^s|u'|^{2s}} 
     \gs \eps_{k+1}^6\log\log N_k$$
  so that (iv) of Theorem~\ref{thm:gen:lower:2} holds 
  provided $C$ was chosen large enough at the beginning.  
It follows that $$\Hdim \Sing(2)\ge\frac{4}{3}$$ 
  and this completes the proof of Theorem~\ref{thm:main}.

\section{Slowly divergent trajectories}\label{S:Starkov}
In this section, we prove 
\begin{theorem}\label{thm:Starkov}
Given any function $W(t)\to-\infty$ as $t\to\infty$ there exists a dense set 
  of $\bx\in\Sing^*(2)$ with the property $W_\bx(t)\ge W(t)$ for all sufficiently 
  large $t$.  
\end{theorem}
This answers affirmatively a question of A.N.~Starkov \cite{St} concerning the 
  existence of slowly divergent trajectories for the flow on $\SL_3\R/\SL_3\Z$ 
  induced by $g_t$.  

\begin{lemma}\label{lem:f}
Given $\delta>0$ and a function $F(t)\to\infty$ as $t\to\infty$ there exists 
  $t_0>0$ and a monotone function $f(t)\to\infty$ as $t\to\infty$ such that 
\begin{enumerate}
  \item[(i)] $f(t)\le F(t)$ for all $t>t_0$, and 
  \item[(ii)] $f(t+f(t))\le f(t)+\delta$ for all $t>t_0$.  
\end{enumerate}
\end{lemma}
\begin{proof}
We may reduce to the case where $F(t)$ is a nondecreasing function.  
Let $t_0$ be large enough so that $y_0=F(t_0)>0$ and for $k>0$ set 
  $t_k=t_{k-1}+y_{k-1}$ and $y_k=\min\left(F(t_k),y_{k-1}+\delta\right)$.  
Since $y_k\ge y_0>0 \forall k$ we have $t_k\to\infty$ and therefore 
  also $y_k\to\infty$.  
Let $f(t)=y_k$ for $t_k\le t<t_{k+1}$ so that $f(t)=y_k\le F(t_k)\le F(t)$ 
  since $t_k\le t$ and $F(t)$ is nondecreasing.  
Moreover, $t_{k+1}=t_k+y_k\le t+f(t)<t_{k+1}+y_k\le t_{k+1}+y_{k+1}=t_{k+2}$ 
  so that $f(t+f(t))=y_{k+1}\le y_k+\delta = f(t) + \delta$.  
\end{proof}

\begin{definition}\label{def:tau(v)}
For any $v\in Q$, let 
  $$\tau(v):=-\frac{1}{3}\log\frac{|L(v)|}{|v|^2}=-\frac{1}{2}\log\frac{\eps(v)}{|v|}.$$  
\end{definition}

\begin{lemma}\label{lem:eps:tau}
There exists $C>0$ such that for any $0<\eps'<1$ and any $u\in Q$, 
  there exists $u'\in\cN_{\eps'}(u)$ such that 
\begin{enumerate}
  \item[(i)] $~|\log \eps(u')-\log \eps'|\le C$, and 
  \item[(ii)] $\big|\tau(u')-\tau(u)-2|\log\eps'|-|\log\eps(u)|\big|\le C$.  
\end{enumerate}
\end{lemma}
\begin{proof}
Let $L'=\Hat L(u)$ and let $u'\in Q$ be determined by $L'=u\wedge u'$ and 
  $$|u'|>(\eps')^{-3}|L'|^2 \ge |u'|-|u|.$$  
Then $u'\in\cN_{\eps'}(u)$ by the first inequality.  
By Lemma~\ref{lem:Hat:L} and (\ref{ieq:eps<2}) 
  $$|u'| > |L'|^2 \ge \frac{|u|^2}{|L(u)|^2} > 2|u|$$ 
  so that $\eps(u')^3\asymp\frac{|L'|^2}{|u'|}\asymp(\eps')^3$, giving (i).  
Since 
\begin{align*}
  \tau(u')-\tau(u) &= \frac{1}{2}\log\frac{|u'|}{|u|} 
                      - \frac{1}{2}\log\frac{\eps'}{\eps(u)} + O(1) \\
                   &= 2|\log \eps'| + |\log \eps(u)| + O(1)
\end{align*}
  (ii) follows.  
\end{proof}

\begin{proof}[Proof of Theorem~\ref{thm:Starkov}]
Let $\Tilde f$ be the function obtained by applying Lemma~\ref{lem:f} with 
  $F=-W(t)$ and some given $\delta>0$ to be determined later.  
Set $f=3^{-1}\Tilde f$ and note that $f$ satisfies 
\begin{enumerate}
  \item[(i)] $3f(t)\le-W(t)$ for all $t>t_0$, and 
  \item[(ii)] $f(t+3f(t))\le f(t)+\delta$ for all $t>t_0$ 
\end{enumerate}
  and since $f(t)\to\infty$, given any $A>0$ we can choose $t_0$, perhaps 
  even larger, so that, in addition to (i) and (ii), $f$ also satisfies 
\begin{enumerate}
  \item[(iii)] $f(t+3f(t)+A)\le f(t)+2\delta$ for all $t>t_0$.  
\end{enumerate}
We claim there is a constant $B$ such that for any $u\in Q_1$ satisfying 
\begin{equation}\label{ieq:u}
  |f(\tau(u)) + \log \eps(u)| \le B
\end{equation}
  and such that $|u|$ larger than some constant depending only on $f$ 
  there exists $u'\in\Hat\sigma_1(u)$ such that 
  $$|f(\tau(u')) + \log \eps(u')| \le B.$$  
Indeed, given $u$ satisfying (\ref{ieq:u}), we let $u'$ be obtained by applying 
  Lemma~\ref{lem:eps:tau} with $\eps'<1$ determined by 
  $$|\log\eps'| = f\big(\tau(u)+|\log\eps(u)|\big).$$  
Then, if $A\ge3B$ we have 
\begin{align*}
  |\log\eps'| &\le f(\tau(u)) + 3f(\tau(u)) + 3B \\
              &\le f(\tau(u)) + 2\delta \\
              &\le |\log\eps(u)| + B + 2\delta.  
\end{align*}
By Lemma~\ref{lem:eps:tau}, 
\begin{align*}
  \tau(u') &\le \tau(u) + 2|\log\eps'| + |\log\eps(u)| + C \\
           &\le \tau(u) + 3|\log\eps(u')| + 2B + C + 4\delta \\
           &\le \tau(u) + 3f(\tau(u)) + 5B + C + 4\delta 
\end{align*}
  so that if $A\ge 5B + C + 4\delta$ we have 
\begin{align*}
  f(\tau(u')) &\le f(\tau(u)) + 2\delta \\
              &\le |\log\eps'| + 2\delta \\
              &\le |\log\eps(u')| + C + 2\delta.  
\end{align*}
Now, $$|\log\eps'| \ge f(\tau(u)) \ge |\log\eps(u)| - B$$ so that 
  $$\tau(u') \ge \tau(u) + 3|\log\eps(u)| - 2B - C.$$  
Assuming $|u|$ large enough so that $3|\log\eps(u)|\ge 2B+C$ we have 
\begin{align*}
  f(\tau(u')) &\ge f(\tau(u)) \\
              &\ge f\big(\tau(u) + 3f(\tau(u)) + A\big) - 2\delta \\
              &\ge f\big(\tau(u) + 3|\log\eps(u)| + A - 3B\big) - 2\delta \\
              &\ge |\log\eps'| - 2\delta \\
              &\ge |\log\eps(u')| - C - 2\delta.  
\end{align*}
Setting $A=6C+14\delta$, we see that the claim follows with $B=C+2\delta$.  

Given any nonempty open set $U\subset\R^2$, we can choose $u_0\in Q$ 
  such that $\Delta(u_0)\subset U$.  
Indeed, choose any $\bx_0\in U\setminus\Q^2$ and let $u_0\in\Sigma(\bx_0)$ 
  be such that $|u_0|$ is large enough so that $\Delta(u_0)\subset U$.  
Let $\delta$ be chosen large enough at the beginning so that (\ref{ieq:u}) 
  holds for $u=u_0$.  
Let $\Sigma_0=(u_k)$ be a sequence constructed by recursive definition using 
  the claim.  
Since 
\begin{equation}\label{ieq:tau}
  \tau(u_{k+1}) = \tau(u_k) + 3|\log\eps(u_k)| + O(1)
\end{equation}
  and $\eps(u_k)\asymp \exp(-f(\tau(u_k)))$ by construction, by choosing 
  $|u_0|$ large enough initially we can ensure that $\eps(u_k)<\tfrac13$ 
  for all $k$ so that $\tau(u_k)$ increases to infinity as $k\to\infty$.  
Since $f(t)\to\infty$ as $t\to\infty$, this implies $\eps(u_k)\to0$ as 
  $k\to\infty$.  
By construction, $u_{k+1}\in\cN_{\eps_k}(u_k)$ so that Theorem~\ref{thm:suff}(c)  
  implies $$\bx:=\lim_k\Dot u_k\in\Sing(2).$$  
If $\bx$ lies on a rational line, then $W_\bx(t)\le-\frac{t}{2}+C$ for 
  some constant $C$ and all large enough $t$.  
It is clear that we could have, at the start, reduced to the case where, 
  say, $W(t)>-\log t$ for all $t$, so that $\bx\in\Sing^*(2).$  

Let $D=|\log(1-3^{-6})|$.  
Theorem~\ref{thm:suff}(d) implies for all $t\in[\tau(u_k),\tau(u_{k+1})]$ 
\begin{align*}
  -W_\bx(t) &\le 3|\log\eps(u_k)| + D \\
            &\le 3f(\tau(u_k)) + 3B + D\\
            &\le -W(\tau(u_k)) + 3B + D\\
            &\le -W(t) + 3B + D.  
\end{align*}
It is clear that we could have chosen $f$ initially to satisfy 
  (i') $3f(t)\le -W(t) - 3 B - D$ for all $t>t_0$ instead of (i).  
With this choice, we conclude $W_\bx(t)\ge W(t)$ for all $t>t_0$.  
\end{proof}


\begin{thebibliography}{9999}
\bibitem{Ba77}
  Baker, R. C. 
  {\em Singular $n$-tuples and Hausdorff dimension,} 
  Math. Proc. Cambridge Philos. Soc., {\bf 81} (1977), no. 3, 377--385.  


\bibitem{Ba92}
  Baker, R. C. 
  {\em Singular $n$-tuples and Hausdorff dimension II,} 
  Math. Proc. Cambridge Philos. Soc., {\bf 111} (1992), no. 3, 577--584.  

\bibitem{Ch03}
  Cheung, Y.  
  {\em Hausdorff dimension of the set of nonergodic directions, 
  (with an appendix by M.~Boshernitzan)} 
  Ann. of Math., {\bf 158} (2003), 661--678.  

\bibitem{Ch07}
  Cheung, Y. 
  {\em Hausdorff dimension of the set of points on divergent trajectories 
    of a homogeneous flow on a product space,} 
  Ergod. Th. Dynam. Sys., {\bf 27} (2007), 65--85.  

\bibitem{Da85}
  Dani, S. G. 
  {\em Divergent trajectories of flows on homogeneous spaces and 
        Diophantine approximation,} 
  J. Reine Angew. Math. {\bf 359} (1985), 55--89.  
  (Corrections: J. Reine Angew. Math. {\bf 360} (1985), 214.)  

\bibitem{Da86}
  Dani, S. G. 
  {\em Bounded orbits of flows on homogeneous spaces,}
        Comment. Math. Helv. {\bf 61} (1986), no. 4, 636--660.  

\bibitem{Fa03}
  Falconer, K.
  {\em Fractal geometry, Mathematical foundations and applications, 2nd ed.}
  John Wiley \& Sons Inc., 2003.  

\bibitem{Kh26}
  Khintchine, A. 
  {\em Zur metrischen Theorie der Diophantischen Approximationen,}
  Math. Z. {\bf 24} (1926), 706–-714.  


\bibitem{KM96}
  Kleinbock, D. Y. and Margulis, G. A.
  {\em Bounded orbits of nonquasiunipotent flows on homogeneous spaces.}
        Sina\u\i's Moscow Seminar on Dynamical Systems, 141--172, 
        Amer. Math. Soc. Transl. Ser. 2, 171, 
        Amer. Math. Soc., Providence, RI, 1996. 

\bibitem{KW1}
  Kleinbock, D. and Weiss, B. 
  {\em Dirichlet's problem on Diophantine approximation and homogeneous flows,} 
        Algebraic and Topological Dynamics, 
        Contemp. Math., {\bf 385} (2005) AMS, Providence, RI, 2005, 281--292. 

\bibitem{KW2}
  Kleinbock, D. and Weiss, B. 
  {\em Friendly measures, homogeneous flows and singular vectors,} 
        Journal of Modern Dynamics, to appear.  

\bibitem{KW04}
  Kleinbock, D. and Weiss, B. 
  {\em Bounded geodesics in moduli space,} 
        Int. Math. Res. Not. 2004, no. 30, 1551--1560.

\bibitem{JL1}
  Lagarias, J.C.
  {\em Best Simultaneously Diophantine Approximations. I. 
        Growth Rates of Best Approximation Denominators,}
        Trans. Amer. Math. Soc., {\bf 272} (1982), no. 2, 545--554.  

\bibitem{JL2}
  Lagarias, J.C.
  {\em Best Simultaneously Diophantine Approximations. II. 
        Behavior of consecutive best approximations,}  
        Pacific J. Math., {\bf 102} (1982), no. 1, 61--88.  

\bibitem{JL3}
  Lagarias, J.C.
  {\em Geodesic Multidimensional Continued Fractions,}
        Proc. London Math. Soc., {\bf 69} (1994), no. 3, 464--488.  

\bibitem{Ma92} 
  Masur, H. 
  {\em Hausdorff dimension of the set of nonergodic foliations of a 
    quadratic differential,} 
  Duke Math. J., {\bf 66} (1992), 387--442.  

\bibitem{Mc87}
  McMullen, C.
  {\em Area and Hausdorff dimension of Julia sets of entire functions,}
  Trans. Amer. Math. Soc., {\bf 300} (1987), 329--342.  

\bibitem{MS91}
  Masur, H. and Smillie, J.
  {\em Hausdorff dimension of sets of nonergodic measured foliations,} 
        Ann. of Math., {\bf 134} (1991), 455--543.

\bibitem{Ry90}
  Rynne, B. P. 
  {\em A lower bound for the Hausdorff dimension of sets of singular 
     $n$-tuples,} 
  Math. Proc. Cambridge Philos. Soc., {\bf 107} (1990), no. 2, 387--394.  


\bibitem{Sc66}
  Schmidt, W. M. 
  {\em On badly approximable numbers and certain games,}
  Trans. Amer. Math. Soc., {\bf 123} (1966), 178-–199. 



\bibitem{St}
  Starkov, A. N. 
  {\em Dynamical systems on homogeneous spaces.}  
        Translated from the 1999 Russian original by the author. 
        Translations of Mathematical Monographs, 190. 
        American Mathematical Society, Providence, RI, 2000.  

\bibitem{W1}
  Weiss, B.
  {\em Divergent trajectories on noncompact parameter spaces,}
       Geom. Funct. Anal., {\bf 14} (2004), no. 1, 94--149.  

\bibitem{W2}
  Weiss, B.
  {\em Divergent trajectories and $\Q$-rank,}
       Israel J. Math., {\bf 52} (2006), 221-227.  

\bibitem{Ya87}
  Yavid, K. Yu. 
  {\em An estimate for the Hausdorff dimension of a set of singular vectors,}
  Dokl. Akad. Nauk BSSR, {\bf 31} (1987), no. 9, 777--780, 859.  

\end{thebibliography}
\end{document}